\documentclass[10pt]{extarticle}
\usepackage{graphicx}
\usepackage{amssymb}
\usepackage{amsfonts,amsmath,amsxtra,amsthm,amssymb,latexsym}
\usepackage[cp1251]{inputenc}
\usepackage{color,xcolor,braket,graphicx}

\newtheorem{theorem}{Theorem}[section]

\newtheorem{lemma}[theorem]{Lemma}
\newtheorem{corollary}[theorem]{Corollary}
\newtheorem{remark}[theorem]{Remark}

\newcommand{\card}{\mbox{card}}

\makeatletter

\renewcommand{\@biblabel}[1]{#1.}

\makeatother \hfuzz=0.5pt \tolerance=500
\begin{document}

\begin{center}
{	{ \bf On Optimal Recovery and Information Complexity\\ in Numerical Differentiation and Summation}}

\end{center}

\vspace*{5mm}
\centerline{\textsc { Y.V. Semenova $\!\!{}^{\dag}$, S.G. Solodky  $\!\!{}^{\dag,\ddag}$} }

%\centerline{\blue{$\!\!{}^{\ddag}\!\!$ University of Giessen, Department of Mathematics, Giessen, Germany}}

%\vspace*{10mm} \centerline{\textsc {Semenova Y.V. $^{\dag}$, Solodky S.G.$^{*, \dag}$}}
%
\vspace*{5mm}
\centerline{$\!\!{}^{\dag}\!\!$ Institute of Mathematics, National Academy of Sciences of Ukraine, Kyiv}
\centerline{$\!\!{}^{\ddag}\!\!$ University of Giessen, Department of Mathematics, Giessen, Germany}

\vspace{4mm}
%\begin{small}
%	\begin{quote}
%		\textsc{Abstract.}
%		
%		
%	\end{quote}
%\end{small}
%
%
%\begin{small}
%	\begin{quote}
%		The problem of optimal recovering high-order mixed derivatives of bivariate functions is studied.
%		For numerical differentiation, a new version of the truncation method is proposed.
%		It is established that this algorothm is order-optimal
%		both in the sense of accuracy and in terms of the amount of involved discrete information.
%	\end{quote}
%\end{small}

%\begin{indented}
%\item[]Juny 2024 (minor update March 2024)
%\end{indented}

\begin{abstract}
{ In this paper, we study optimization problems of numerical differentiation and summation methods on classes of univariate functions. Sharp estimates (in order) of the optimal recovery error and information complexity are calculated for these classes. Algorithms are constructed based on the truncation method and Chebyshev polynomials to implement these estimates. Moreover, we establish under what conditions the summation problem is well-posed.
}
\vspace*{3mm}

{\textit{Key words:} Numerical differentiation, numerical summation, Chebyshev polynomials, truncation method, information complexity, optimal recovery}
\end{abstract}

%

%\maketitle

\section{Introduction. Description of the problem}
It is known (see, for example, \cite{MasHand}, \cite {COH2007}) that in solving many problems of mathematical physics and
engineering, it seems convenient to search for a solution in the form of a finite sum over orthogonal polynomials.
In this case, the proposed methods are reduced to recovering   Fourier coefficients
of the solution itself and its derivatives
(see, for instance, \cite{Ar2014}, \cite{Pies1972}).
In turn, this leads to problems of numerical differentiation and summation (NDS).
Generally speaking, these problems are unstable to small perturbations of the input data and require
regularization when constructing approximations. Since the emergence of Tikhonov's regularization theory
in the 60s of the last century, these tasks have attracted the attention of a wide range of researchers.

Among the works devoted to the construction and justification of effective methods of NDS,
we highlight
\cite{Ahn&Choi&Ramm_2006}, \cite{Dolgopolova&Ivanov_USSR_Comput_Math_Math_Phys_1966_Eng}, \cite{Cul71}, \cite{And84},  \cite{EgorKond_1989},   \cite{Groetsch_1992_V74_N2}, \cite{Hanke&Scherzer_2001_V108_N6}, \cite{Lu&Naum&Per},  \cite{Mathe_Per},\, \cite{Meng&Zhaoa&Mei&Zhou_2020},\,
\cite{Nakamura&Wang&Wang_2008},\, \cite{PoPo}, \cite{Qian&Fu&Xiong&Wei_2006}, \cite{Qu96}, \cite{Ramm_1968_No11},  \cite{RammSmir_2001}, \cite{Sol_Shar},\, \cite{VasinVV_1969_V7_N2}, \cite{Wang_Hon_Ch_2006}, \cite{WW2005}, \cite{Zhao_2010}, \cite{Zhao&Meng&Zhao&You&Xie_2016}.
%\cite{Dolgopolova&Ivanov_USSR_Comput_Math_Math_Phys_1966_Eng}, \cite{Ramm_1968_No11}, ,
%\cite{Groetsch_1992_V74_N2}, \cite{Hanke&Scherzer_2001_V108_N6}, ,
% \cite{Zhao_2010}, , \cite{EgorKond_1989},  \cite{ Wang_Hon_Ch_2006},, \cite{Nakamura&Wang&Wang_2008}.
As is known,  in assessing the effectiveness of approximate methods, their optimality plays an important role.
At the same time, in NDS, until recently the optimality of methods was understood only
in the sense of their accuracy.
The point is that the optimal accuracy can be achieved by using different amounts of discrete information.
Therefore, it makes sense to also study information complexity of NDS.
In other words, it is very important to research NDS methods that achieve
optimal accuracy by using the minimal possible amount of discrete input data.
Such input data is often Galerkin information, which is traditionally understood as
a set of values of Fourier coefficients.
It should be noted that in their recent publications, the authors investigated the information complexity
of NDS in cases where inaccurately known values of the Fourier coefficients w.r.t.
trigonometric system \cite{Sol_Stas_JC2020}, \cite{SSS_CMAM}, \cite{Sol_Stas_JC2023},
as well as noisy values of the Fourier-Legendre coefficients \cite{Sem_Sol_2021}, \cite{Sol_Stas_UMZ2022}, \cite{Sem_Sol_ETNA},
were considered as Galerkin information.
Continuing with these studies, this article examines the problem of optimizing NDS methods in a situation where the input data is in the form of perturbed Fourier-Chebyshev coefficients. Moreover, it will be shown that, in the $C$-metric, Chebyshev polynomials provide better accuracy for the differentiation problem compared to Legendre polynomials. In addition, we will establish which parameters affect the stability of the numerical summation problem.

For the further presentation of the material, we need the following notation and concepts.
Let $\{\bar{T}_k(t)\}_{k=0}^\infty$  be the system of Chebyshev polynomials of the first kind as
$$
\bar{T}_k(t) = \cos(k\arccos t)
=\frac{(-1)^k 2^{k} k!}{(2k)!}\sqrt{1-t^2} \frac{d^k}{dt^k}(1-t^2)^{k-1/2}, \quad k\in \mathbb{N}_0 ,
$$
where $\mathbb{N}_0=\{0\}\bigcup\mathbb{N}$.
{We proceed from $\{\bar{T}_k(t)\}$ to an orthonormal system $\{T_k(t)\}$ of Chebyshev polynomials
according to the relations}
%As is known,  $\{\bar{T}_k\}$ generates an orthonormal system of Chebyshev polynomials
$$
T_0(t)=\frac{1}{\sqrt{\pi}} \bar{T}_0(t),\qquad
T_k(t)=\frac{\sqrt{2}}{\sqrt{\pi}} \bar{T}_k(t),\quad k\in \mathbb{N} .
$$
By $L_{2,\omega}$ we mean a weighted Hilbert space of  real-valued functions $f(t)$ which are square-summable on $[-1,1]$ with the weight $\omega(t)=(1-t^2)^{-1/2}$.
In $L_{2,\omega}$ the inner product and norm are introduced in the standard way
$$
\langle f, g\rangle=\int_{-1}^{1} \omega(t) f(t) g(t) \ d t,
$$
$$
\|f\|_{2,\omega}^2 := \int_{-1}^1 \omega(t) |f(t)|^2 \, d t =
\sum_{k=0}^{\infty}|\langle f, T_{k} \rangle|^2 < \infty .
$$
Here $$ \langle f, T_{k}\rangle=\int_{-1}^{1} \omega(t) f (t) T_k(t) dt, \quad
k=0,1,2,\ldots,
$$
are Fourier-Chebyshev coefficients of $f$.
Moreover, let $C$ be the space of continuous on $[-1,1]$ real-valued functions
and $\ell_p$, $1\leq p\leq\infty$, be the vector space of sequences of real numbers
%the space of numerical sequences
$\overline{x}=\{x_{k}\}_{k\in\mathbb{N}_0}$, such that the corresponding relation
$$
%\begin{equation}\label{noise}
\|\overline{x}\|_{\ell_p}  := \left\{
\begin{array}{cl}
	\bigg(\sum\limits_{k\in\mathbb{N}_0} |x_{k}|^p\bigg)^{\frac{1}{p}} < \infty ,
	\ & 1\leq p<\infty ,
	\\\\
	\sup\limits_{k\in\mathbb{N}_0}  |x_{k}| < \infty ,
	\ & p=\infty ,
\end{array}
\right.
$$
is fulfilled.

We introduce the space of functions
$$
%\begin{equation}\label{norm}
W_{s}^\mu =\{f\in L_{2,\omega}: \quad \|f\|_{s,\mu}^s :=\sum_{k=0}^{\infty} ({\underline{k}})^{s\mu}|\langle f,
T_{k}\rangle|^s<\infty\},
%\end{equation}
$$
where $\mu>0$,\ $1\le s<\infty$,\ $\underline{k}=\max\{1,k\}$,\ $k=0,1,2,\dots$. Note that in the future we will use
the same notations both for space and for a unite ball from this space: $W_{s}^{\mu} = \{f\in
W_{s}^{\mu}\!: \|f\|_{s,\mu} \leq 1\}$, what we call a class of functions. What exactly is meant by $
W_{s}^{\mu}$, space or class, will be clear depending on the context in each case.

\vskip 4mm

\begin{remark} \label{Tom}

Let $\widehat{W}^r_{2}$, $r=1,2,\ldots$, denote the set of all functions $f(t)$ for which $f, f',\ldots, f^{(r-1)}$
are absolutely continuous on any interval $(\eta_1,\eta_2)$, $\eta_1>-1$, $\eta_2<1$, and the condition
$$\int_{-1}^{1} \omega(t)^{-2r+1} |f^{(r)}(t)|^2 \ d t < \infty $$
is satisfied.
Then from \cite{Tom} it follows that $\widehat{W}^r_{2}$ can be equipped by the norm
$$
\|f\|_{\widehat{W}^r_{2}} := \big(\sum_{j=0}^r \int_{-1}^{1} \omega(t)^{-2j+1} |f^{(j)}(t)|^2 dt\big)^{1/2}
$$
and the following two statements are equivalent:
1) $f\in \widehat{W}^r_{2}$,\, 2) $f\in W^r_{2}$.
\end{remark}

\vskip 4mm

We represent a function $f(t)$ from $W_{s}^{\mu}$ as
%$$
\begin{equation}\label{function}
	f(t) =
%\frac{1}{\pi} \langle f, T_{0}\rangle T_0(t) + \frac{2}{\pi}
\sum_{k=0}^{\infty}\langle f, T_{k}\rangle T_k(t),
\end{equation}
%$$
and by its derivative of order $r$, $r\in \mathbb{N}$, we mean the following series
\begin{equation}\label{r_deriv}
	f^{(r)}(t) = \sum_{k=r}^{\infty}\langle f, T_{k}\rangle T^{(r)}_k(t).
\end{equation}

We will investigate problems of recovering $f^{(r)}$ using perturbed input data.
%the perturbed values
%of the Fourier-Chebyshev coefficients of the function $f$ as discrete information.
Further, we give a strict statement of the problems to be studied.
Suppose that instead of the function $f\in W_s^\mu$ we only know some its approximation $f^\delta \in L_{2,\omega}$
such that $\|f - f^\delta\|_{2,\omega} \le \delta$, $0< \delta <1$.
Let $X=L_{2,\omega}$ or $X=C$. As a {\it solution operator} of the NDS problem, we take any mapping.
$$
m^{(r)} : L_{2,\omega} \to X ,
$$
where the element $m^{(r)} f^\delta \in X$ serves as an approximate solution of the problem.
We denote by ${\mathcal M}$ the set of all mappings from $L_{2,\omega}$ into $X$.
Mappings from ${\mathcal M}$, in general, do not necessarily  require linearity, continuity, or even stability.
Such a general understanding of the solution operator is explained by the desire to compare as wide a range of possible methods of differentiation as possible.

By the {\it optimal recovery error}, we will understand the quantity
$$
E^{(r)}_{\delta} (W^{\mu}_{s}, {\mathcal M}, X) =
\inf\limits_{m^{(r)}\in {\mathcal M}}
\ \sup_{\substack{\|f\|_{s,\mu}\leq 1}}
 \ \sup_{\substack{f^\delta \in L_{2,\omega}, }
 \atop \|f-f^\delta\|_{2,\omega} \leq \delta }
 \| f^{(r)} - m^{(r)} f^{\delta}\|_{X} .
$$
The quantity $E^{(r)}_{\delta}$ characterizes the highest possible accuracy of approximation of
the derivative $f^{(r)}\in W_s^\mu$,
which can be achieved by all possible mappings from $L_{2,\omega}$ into $X$.
A similar quantity is studied for various problems, for example, in \cite{OsJC}.
It answers the question about the minimal approximation error but leaves out the question
of the smallest amount of discrete information needed to achieve a given accuracy.
These kinds of problems are investigated in  the IBC (Information-Based Complexity) theory,
the foundations of which were laid in monographs \cite{TrW} and \cite{TrWW}.

To study the information complexity of the numerical differentiation problem,
we need to introduce additional concepts and definitions into consideration.

Let us now suppose that the error level is measured in the metric of $\ell_p$.
More precisely, we assume that there is a sequence of real numbers $\overline{f^\delta}= \{\langle
f^\delta, T_{k} \rangle\}_{k\in\mathbb{N}_0}$ such that for $\overline{\xi}=
\{\xi_{k}\}_{k\in\mathbb{N}_0}$, where $\xi_{k}=\langle f-f^\delta,T_{k}\rangle$, and for some $1\leq
p\leq \infty$ the relation
\begin{equation}\label{perturbation2}
\|\overline{\xi}\|_{\ell_p} \leq \delta , \quad 0<\delta <1 ,
\end{equation}
% where $\overline{\xi}= \{\xi_{k,j}\}_{k,j\in\mathbb{N}_0}$, $\xi_{k,j}=\langle f-f_\delta,\varphi_{k,j}\rangle$,
%$1\leq p\leq \infty$ .
is true.

By an {\it information operator}
$$
G: L_{2,\omega} \to \ell_p
$$
we mean a mapping that assigns to any function $f\in L_{2,\omega}$ the set of values of its Fourier-Chebyshev coefficients
\begin{equation}\label{FHc}
G (f) = \left(\langle f, T_{0} \rangle, \langle f, T_{1} \rangle,\ldots, \langle f, T_{n} \rangle, \ldots\right) .
\end{equation}
Note that discrete information in the form of values of Fourier coefficients, for example, of the form (\ref{FHc}),
is usually called {\it Galerkin information}.

By a {\it recovery operator}, we understand any mapping
%\begin{equation}\label{rec_op}
$$
\psi^{(r)}: \ell_p \to X ,
$$
%\end{equation}
which assigns to any set of numbers $\overline{f}=(f_0, f_1, \ldots, f_n, \ldots)\in \ell_p$
some element $\psi^{(r)}\overline{f}\in X$.

By a {\it differentiation method}, we mean any operator
\begin{equation}\label{DM}
\psi^{(r)} G: L_{2,\omega} \to X ,
\end{equation}
which has the form of a composition of two mappings: the information operator $G$ and
the recovery operator $\psi^{(r)}$.
The element $\psi^{(r)}G(f^\delta)\in X$ serves as an approximation
to the derivative $f^{(r)}$ of the function $f\in W_s^\mu$.
We denote by $\Psi$ the set of all methods of the form (\ref{DM}).

The only condition of the methods from $\Psi$ is that they use Galerkin information (\ref{FHc}) as discrete information.
In this case, the number of Fourier-Chebyshev coefficients involved can be infinite.

{The error of the method $\psi^{(r)}G$ on the class $W^{\mu}_{s}$ is determined by the quantity
 $$
e_{\delta}(W^{\mu}_{s}, \psi^{(r)}G, X, \ell_p)
 = \sup_{\substack{\|f\|_{s,\mu}\leq 1}}
 \ \sup_{\substack{f^\delta \in L_{2,\omega}: \, \overline{f}^{\delta}=\overline{f}+\overline{\xi}}
 \atop \|\overline{\xi}\|_{\ell_p} \leq \delta }
 \| f^{(r)} - \psi^{(r)}G(f^{\delta}) \|_{X} .
 $$}
We will investigate the quantity
$$
{\cal E}^{(r)}_{\delta} (W^{\mu}_{s}, \Psi, X, \ell_p) =
\inf\limits_{\psi^{(r)}G\in\Psi} e_{\delta}(W^{\mu}_{s}, \psi^{(r)}G, X, \ell_p) ,
$$
which describes the highest possible accuracy that can be achieved
in the numerical differentiation of an arbitrary function $f\in W^{\mu}_{s}$ by methods that use
the values of the Fourier-Chebyshev coefficients, perturbed in the metric $\ell_p$, as discrete information.

In addition, let us consider another quantity that characterizes not only the highest possible accuracy
but also the minimal amount of Galerkin information that ensures this accuracy.
So, on the coordinate axis $[r,\infty)$, we take an arbitrary finite domain $\Omega$, $\card(\Omega) = N$,
where $\card(\Omega)$ means the number of points that make up $\Omega$.

We consider an operator
$$
G_\Omega: L_{2,\omega} \to \ell_p
$$
which operates as follows:
each function $f\in L_{2,\omega}$ is assigned a numerical sequence. The $k$-th component of this sequence coincides with $\langle f, T_k \rangle$, for $k \in \mathbb{N}_0$, if $k \in\Omega$, and it is equal to $0$ otherwise.
Then by algorithm of numerical differentiation, we will understand any operator of the form
$$
\psi^{(r)} G_\Omega: L_{2,\omega} \to X ,
$$

where $\psi^{(r)}: \ell_p \to X$, as before, is a recovery operator.
The set of all algorithms of the form $\psi^{(r)} G_\Omega$ is denoted by $\Psi_\Omega$.
Let $\Psi_N := \cup_{\Omega: \card(\Omega)\le N} \Psi_\Omega$,
where the union is taken over all domains $\Omega$ with $\card(\Omega)\le N$.
It is easy to see that $\Psi_N \subset \Psi$.
The contraction of the set $\Psi_N$ relative to $\Psi$  occurs because
the algorithms in $\Psi_N$ use no more than $N$ values of the Fourier-Chebyshev coefficients.

The {\it minimal radius of the Galerkin information} for the problem of numerical differentiation on the class
$W^{\mu}_{s}$ is given by
%\begin{equation}   \label{radius}
$$
R^{(r)}_{N,\delta} (W^{\mu}_{s}, \Psi_N, X, \ell_p) =
\inf\limits_{\psi^{(r)}G_\Omega\in\Psi_N} e_{\delta}(W^{\mu}_{s}, \psi^{(r)}G_\Omega, X, \ell_p) .
$$

The quantity  $R^{(r)}_{N,\delta} (W^{\mu}_{s}, \Psi_N, X, \ell_p)$ describes the highest possible accuracy,
which can be achieved by numerical differentiation of an arbitrary function $f\in W_{s}^{\mu}$ , while using at most  $N$ values of its Fourier-Chebyshev coefficients that are $\delta$-perturbed in the $\ell_p$-metric.
Note that the minimal radius of Galerkin information
%in the problem of recovering the first partial derivative was studied in \cite{Sol_Stas_UMZ2022}, and
for other types of ill-posed problems was studied in
\cite{PS1996}, \cite{Mileiko_Solodkii_2014}. It should be noted that the Galerkin minimal radius characterises the information complexity of the problem under consideration and is traditionally studied within the framework of IBC theory.

Due to the obvious relation $\Psi_N \subset \Psi \subset {\mathcal M}$, for any $N$, $\delta>0$ it is true
\begin{equation}  \label{IV}
R_{N,\delta}^{(r)}(W^{\mu}_{s}, \Psi_N, X, \ell_2)
\geq {\cal E}_{\delta}^{(r)}(W^{\mu}_{s}, \Psi, X, \ell_2)
\geq E_{\delta}^{(r)}(W^{\mu}_{s}, {\mathcal M}, X) ,
\end{equation}
and for any $N$, $\delta>0$ and $2\le p\le \infty$ we have
\begin{equation}  \label{V}
R_{N,\delta}^{(r)}(W^{\mu}_{s}, \Psi_N, X, \ell_p)
\geq {\cal E}_{\delta}^{(r)}(W^{\mu}_{s}, \Psi, X, \ell_p) .
\end{equation}

Our research aims to calculate sharp (in order) estimates of the quantities $R_{N,\delta}^{(r)}$,
$E_{\delta}^{(r)}$, ${\cal E}_{\delta}^{(r)}$, and to construct methods to implement these estimates.

\section{Truncation method. Error estimate in the metric of $C$} \label{up_bound_C}

To solve the problems described above, we propose to use the truncation method.
The essence of this method is to replace the Fourier
series (\ref{r_deriv}) with a finite Fourier sum using perturbed data $\langle f^\delta, T_{k} \rangle$. In the
truncation method to ensure the stability of the approximation and achieve the required order accuracy,  it is
necessary to choose properly the discretization parameter, which here serves as a regularization parameter. So, the
process of regularization in the method under consideration consists in matching the discretization parameter with the
perturbation level $\delta$ of the input data. The simplicity of implementation is one of the important advantages of this method.

Let $r=1,2,\ldots$. For recovering $r$-derivative of functions $f\in W^{\mu}_{s}$ the proposed truncation method has the form
\begin{equation} \label{ModVer}
	\mathcal{D}_N^{(r)} f^\delta(t) = \sum_{k=r}^{N} \langle f^\delta, T_{k}\rangle
	T^{(r)}_k(t)
\end{equation}
with $\Omega=[r,N]$.

The approximation properties of the method (\ref{ModVer}) will be investigated in Sections
\ref{up_bound_C} and \ref{up_bound_L2}.  In Sections \ref{opt_C} and \ref{opt_L2},
it will be established that the method (\ref{ModVer}) is order-optimal in the sense of the quantities
$R_{N,\delta}^{(r)}$, $E_{\delta}^{(r)}$, ${\cal E}_{\delta}^{(r)}$.
Section \ref{num_prob} is devoted to the problem of numerical summation.
In Section \ref{num_exp}, numerical experiments will be provided to confirm the effectiveness of the proposed approach.

Let us write the error of the method (\ref{ModVer}) as
\begin{equation}\label{fullError}
	f^{(r)}(t)-\mathcal{D}_N^{(r)} f^\delta(t)= \left(f^{(r)}(t)-\mathcal{D}_N^{(r)}
	f(t)\right)+\left(\mathcal{D}_N^{(r)} f(t)-\mathcal{D}_N^{(r)} f^\delta(t)\right).
\end{equation}
Here
$$
\mathcal{D}_N^{(r)} f(t) = \sum_{k=r}^{N} \langle f, T_{k}\rangle T^{(r)}_k(t) .
$$
The parameter $N$ in (\ref{ModVer}) should be chosen depending on
$\delta$, $p$, $s$ and $\mu$ so as to minimize the error estimate for $\mathcal{D}_N^{(r)}$.

It is known (see  \cite[p.34]{MasHand}) that
\begin{equation}\label{MH}
	\frac{d}{d t} \bar{T}_{k}(t) \, = 2 \, k
	\mathop{{\sum}^*}\limits_{l=0}^{k-1} \bar{T}_{l}(t) ,
	\ \ \ k\in\mathbb{N} ,
\end{equation}
where in aggregate \quad $ \mathop{{\sum}^*}\limits_{l=0}^{k-1} \bar{T}_{l}(t)$ the summation
is extended over only those terms for which $k+l$ is odd. Moreover, it holds true
\begin{equation}\label{max}
\max_{-1\le t\le 1} |\bar{T}_k(t)| = \bar{T}_k(1) =1, \qquad k\in \mathbb{N}_0 .
\end{equation}
Since in further calculations, we use the orthonormal system $\{T_k\}$, then for convenience we rewrite relations
(\ref{MH}), (\ref{max}) as
\begin{equation}\label{dif_ort}
	\frac{d}{d t} T_{k}(t) \, = 2 \, k
	\mathop{{\sum}^*}\limits_{l=0}^{k-1} \xi_l T_{l}(t) ,
	\quad k\in\mathbb{N} ,\qquad \xi_0=\sqrt{2}, \ \ \xi_l=1,\ \ l\in\mathbb{N} ,
\end{equation}
$$
\max_{-1\le t\le 1} |T_k(t)| = T_k(1) := \left\{
\begin{array}{cl} 1/\sqrt{\pi}, \ & k=0 ,\\\\
\sqrt{2/\pi}, \ & k\in \mathbb{N} .
\end{array}
\right.
$$

In the sequel, we adopt the convention that $c$ denotes a generic positive coefficient,
which can vary from inequality to inequality and does not depend on the values of $\delta$ and $N$.
The notation $A\preceq B$ means $A\le c B$, $A\succeq B$ means $A\ge c B$, $A\asymp B$ means
$A\preceq B\preceq A$.
Moreover, the notation $A\ll B$ means $A=o(B)$, $A\gg B$ means $B=o(A)$.

Let us estimate the error of (\ref{ModVer}) in the metric of $C$. An upper bound for the norm of the first difference
on the right-hand side of (\ref{fullError}) is contained in the following statement.

\begin{lemma}\label{lemma_BoundErrHCC}
	Let $f\in W^\mu_{s}$, $1\leq s< \infty$, $\mu>2r-1/s+1$. Then it holds
	$$
	\|f^{(r)}-\mathcal{D}^{(r)}_N f\|_{C}\leq c\|f\|_{s,\mu} N^{-\mu+2r-1/s+1} .
	$$
\end{lemma}

\begin{proof}
Using (\ref{dif_ort}), we write down the representation
$$
f^{(r)}(t)-\mathcal{D}^{(r)}_N f(t) =
\sum_{k=N+1}^\infty \langle f, T_{k} \rangle T_k^{(r)} (t) =
2 \sum_{k=N+1}^\infty k\, \langle f, T_{k} \rangle \mathop{{\sum}^*}\limits_{l_1=0}^{k-1} \xi_{l_1} T_{l_1}^{(r-1)}(t)
$$
\begin{equation}  \label{first_dif}
	= \cdots = 2^r \sum_{k=N+1}^\infty k\, \langle f, T_{k} \rangle
	\mathop{{\sum}^*}\limits_{l_1=r-1}^{k-1} l_1 \mathop{{\sum}^*}\limits_{l_2=r-2}^{l_1-1} l_2 \ldots
	\mathop{{\sum}^*}\limits_{l_{r-1}=1}^{l_{r-2}-1} l_{r-1}
	\mathop{{\sum}^*}\limits_{l_{r}=0}^{l_{r-1}-1} \xi_{l_r} T_{l_r}(t) .
\end{equation}
Now we can estimate
$$
\|f^{(r)}-\mathcal{D}^{(r)}_N f\|_{C} \leq
2^r \sum_{k=N+1}^\infty k\, |\langle f, T_{k} \rangle |
\mathop{{\sum}}\limits_{l_1=r-1}^{k-1} l_1 \mathop{{\sum}}\limits_{l_2=r-2}^{l_1-1} l_2 \ldots
\mathop{{\sum}}\limits_{l_{r-1}=1}^{l_{r-2}-1} l_{r-1}
\mathop{{\sum}}\limits_{l_r=0}^{l_{r-1}-1} \xi_{l_r} \|T_{l_r}\|_C
$$
$$
\leq
\frac{2^{r+1/2}}{\sqrt{\pi}} \sum_{k=N+1}^\infty k\, |\langle f, T_{k} \rangle |
\mathop{{\sum}}\limits_{l_1=r-1}^{k-1} l_1 \mathop{{\sum}}\limits_{l_2=r-2}^{l_1-1} l_2 \ldots
\mathop{{\sum}}\limits_{l_{r-1}=1}^{l_{r-2}-1} l_{r-1}^2
\leq
c \sum_{k=N+1}^\infty k^{2r}\, |\langle f, T_{k} \rangle | .
$$
Let's start with the case $1< s< \infty$. Then using the H\"{o}lder inequality and the definition of
$W^\mu_{s}$, we continue for $\mu>2r-1/s+1$
$$
\|f^{(r)}-\mathcal{D}^{(r)}_N f\|_{C}\leq
c \sum_{k=N+1}^{\infty}\ k^{2r-\mu}\, k^\mu\, |\langle f, T_{k} \rangle |
$$
$$
\leq c\, \left(\sum_{k=N+1}^{\infty}\ k^{s\mu}\,|\langle f, T_{k} \rangle |^s\right)^{1/s}
\left(\sum_{k=N+1}^{\infty}\ k^{(2r-\mu)\frac{s}{s-1}}\right)^{(s-1)/s}
$$
$$
\leq c \|f\|_{s,\mu}  N^{-\mu+2r-1/s+1} .
$$
In the case of $s=1$ the assertion of Lemma is proved similarly. \vspace{0.1in}
\end{proof}

The following statement contains an estimate for the second difference on the right-hand side of (\ref{fullError}) in
the metric of $C$.
\begin{lemma}\label{lemma_BoundPertHCC}
	Let the condition (\ref{perturbation2}) be satisfied for $1\le p\le \infty$.  Then for arbitrary function $f\in L_{2,\omega}$ it holds
	$$
	\|\mathcal{D}^{(r)}_N f - \mathcal{D}^{(r)}_N f^\delta\|_{C} \leq c \delta N^{2r-1/p+1} .
	$$
\end{lemma}

\begin{proof}
Using (\ref{dif_ort}), we write down the representation
$$
\mathcal{D}^{(r)}_N f(t) - \mathcal{D}^{(r)}_N f^\delta(t)
= \sum_{k=r}^N  \langle f-f^\delta, T_{k} \rangle T_{k}^{(r)}(t)
= 2 \sum_{k=r}^N k\, \langle f-f^\delta, T_{k} \rangle
\mathop{{\sum}^*}\limits_{l_1=0}^{k-1} \xi_{l_1} T_{l_1}^{(r-1)}(t)
$$
\begin{equation}  \label{second_dif}
	= \cdots = 2^r \sum_{k=r}^N k\, \langle f-f^\delta, T_{k} \rangle
	\mathop{{\sum}^*}\limits_{l_1=r-1}^{k-1} l_1 \mathop{{\sum}^*}\limits_{l_2=r-2}^{l_1-1} l_2 \ldots
	\mathop{{\sum}^*}\limits_{l_{r-1}=1}^{l_{r-2}-1} l_{r-1}
	\mathop{{\sum}^*}\limits_{l_r=0}^{l_{r-1}-1} \xi_{l_r} T_{l_r}(t) .
\end{equation}
Let $1< p<\infty$ first. Then, using the H\"{o}lder inequality, we find
$$
\|\mathcal{D}^{(r)}_N f - \mathcal{D}^{(r)}_N f^\delta\|_{C}
\leq \frac{2^{r+1/2}}{\sqrt{\pi}} \sum_{k=r}^N k\, |\langle f-f^\delta, T_{k} \rangle |
\mathop{{\sum}^*}\limits_{l_1=r-1}^{k-1} l_1 \mathop{{\sum}^*}\limits_{l_2=r-2}^{l_1-1} l_2 \ldots
\mathop{{\sum}^*}\limits_{l_{r-1}=1}^{l_{r-2}-1} l_{r-1}^2
$$
$$
\leq c\, \sum_{k=r}^N k^{2r}\, |\langle f-f^\delta, T_{k} \rangle |
\leq c\, \left(\sum_{k=r}^{N}\,|\xi_k|^p\right)^{1/p}
\left(\sum_{k=r}^{N}\ k^{\frac{2rp}{p-1}}\right)^{(p-1)/p}
$$
$$
\leq c \delta N^{2r-1/p+1} ,
$$
which was required to prove.

In the cases of $p=1$ and $p=\infty$, the assertion of Lemma is proved similarly. \vspace{0.1in}
\end{proof}

The combination of Lemmas \ref{lemma_BoundErrHCC} and \ref{lemma_BoundPertHCC} gives
\begin{theorem} \label{Th2}
	Let $f\in W^\mu_{s}$, $1\leq s< \infty$, $\mu>2r-1/s+1$, and let the condition (\ref{perturbation2})
	be satisfied for $1\le p\le \infty$. Then for $N\asymp \delta^{-\frac{1}{\mu-1/p+1/s}}$ it
	holds
	$$
	\|f^{(r)}-\mathcal{D}^{(r)}_N f^\delta\|_{C} \leq c \delta^{\frac{\mu-2r+1/s-1}{\mu-1/p+1/s}}  .
	$$
\end{theorem}

\begin{proof}
Taking into account  Lemmas \ref{lemma_BoundErrHC}, \ref{lemma_BoundPertHC},  from (\ref{fullError})  we get
$$
\|f^{(r)} - \mathcal{D}^{(r)}_N f^\delta\|_{C} \leq \|f^{(r)}-\mathcal{D}_N^{(r)} f\|_{C} +
\|\mathcal{D}^{(r)}_N f - \mathcal{D}^{(r)}_N f^\delta\|_{C}
$$
$$
\leq c\, \|f\|_{s,\mu} N^{-\mu+2r-1/s+1} + c\, \delta N^{2r-1/p+1}
= c\, N^{2r+1} \left(N^{-\mu-1/s} + \delta N^{-1/p}\right) .
$$
Substituting the rule $N\asymp \delta^{-\frac{1}{\mu-1/p+1/s}}$ into the relation above completely proves Theorem.
\end{proof}
\vskip 2mm

\begin{corollary} \label{Cor2}
In the considered problem, the truncation method  $\mathcal{D}^{(r)}_{N}$ (\ref{ModVer})
	achieves the accuracy
	$$
	O\Big(\delta^{\frac{\mu-2r+1/s-1}{\mu-1/p+1/s}}\Big)
	$$
	on the class $W^{\mu}_{s}$, $\mu>2r-1/s+1$, and requires
	$$
	\card([r,N]) \asymp
	N \asymp \delta^{-\frac{1}{\mu-1/p+1/s}}
	$$
	perturbed Fourier-Chebyshev coefficients.
\end{corollary}

\vskip 2mm

%\begin{remark} \label{Rem3}
%\rm The method $\mathcal{D}^{(r,r)}_{n}$ (\ref{ModVer}) was studied earlier (see \cite{Sem_Sol_2021}) for the problem of
%numerical differentiation of functions from
%$L^{\mu}_{s,2}$  in the case of $r=1$ and $p=s=2$. Thus, the results of Theorems \ref{Th1} and \ref{Th2} generalize
%studies \cite{Sem_Sol_2021} for the case of arbitrary $r,p,s$.
%\end{remark}

\vskip 2mm

\section{Error estimate for truncation method in $L_{2,\omega}$--metric}   \label{up_bound_L2}

Now we have to bound the error of the method (\ref{ModVer}) in the metric of $L_{2,\omega}$.
An upper estimate for the norm
of the first difference on the right-hand side of (\ref{fullError}) is contained in the following statement.

\begin{lemma}\label{lemma_BoundErrHC}
	Let $f\in W^\mu_{s}$, $1\leq s< \infty$, $\mu>2r-1/s+1/2$. Then it holds
	$$
	\|f^{(r)}-\mathcal{D}_N^{(r)} f\|_{2,\omega} \leq c\|f\|_{s,\mu} N^{-\mu+2r-1/s+1/2} .
	$$
\end{lemma}

\begin{proof}
We take the presentation (\ref{first_dif}) for the first difference of (\ref{fullError})
$$
f^{(r)}(t)-\mathcal{D}^{(r)}_N f(t)
$$
$$
= 2^r \sum_{k=N+1}^\infty k\, \langle f, T_{k} \rangle
\mathop{{\sum}^*}\limits_{l_1=r-1}^{k-1} l_1 \mathop{{\sum}^*}\limits_{l_2=r-2}^{l_1-1} l_2 \cdots
\mathop{{\sum}^*}\limits_{l_r=0}^{l_{r-1}-1} \xi_{l_r} T_{l_r}(t)
$$
and will change the order of summation in it.
Recall that  in this representation only those terms  take place for which all  indexes
$l_1+k, l_2+l_1,...,l_r+l_{r-1}$ are odd . Such rule valid also for the term
$\mathcal{D}_N^{(r)} f-\mathcal{D}_N^{(r)} f^\delta$ (see (\ref{fullError})) and the right-hand side of relation (\ref{RHS}).
In the following, for simplicity, we will omit the symbol "*"\, when denoting such summation operations, while taking into account this rule in the calculations.

Using the formula (\ref{first_dif}), we have
$$
f^{(r)}(t)-\mathcal{D}^{(r)}_N f(t)
= 2^r\, \mathop{{\sum}}\limits_{l_{r}=0}^{N-r+1} \xi_{l_r} T_{l_r}(t)
\sum_{k=N+1}^\infty k\, \langle f, T_{k} \rangle B^r_{k}
$$
$$
+ 2^r\, \mathop{{\sum}}\limits_{l_{r}=N-r+2}^{\infty} T_{l_r}(t)
\sum_{k=l_r+r}^\infty k\, \langle f, T_{k} \rangle B^r_{k} ,
$$
where
$$
B^r_{k}:=\mathop{{\sum}}\limits_{l_1=l_r+r-1}^{k-1} l_1 \mathop{{\sum}}\limits_{l_2=l_r+r-2}^{l_1-1}
l_2 \ldots \mathop{{\sum}}\limits_{l_{r-2}=l_r+2}^{l_{r-3}-1}  l_{r-2}\,
\mathop{{\sum}}\limits_{l_{r-1}=l_r+1}^{l_{r-2}-1}  l_{r-1} ,
$$
\begin{equation}   \label{Br}
B^r_{k} = c^* k^{2r-2}
\end{equation}
and the factor $c^*$ does not depend on $N$, $k$.

At first, we consider the case $1<s<\infty.$
Using H\"{o}lder inequality  and the definition of $W^\mu_{s}$,  for $\mu>2r-1/s+1/2$ we get
$$
\|f^{(r)}-\mathcal{D}^{(r)}_N f\|^2_{2,\omega}
\le c\, \mathop{{\sum}}\limits_{l_{r}=0}^{N-r+1}
\left(\sum_{k=N+1}^\infty k\, \langle f, T_{k} \rangle B^r_{k}\right)^2
$$
$$
+ c\, \mathop{{\sum}}\limits_{l_{r}=N-r+2}^{\infty}
\left(\sum_{k=l_r+r}^\infty k\, \langle f, T_{k} \rangle B^r_{k}\right)^2
$$
$$
\leq c  \mathop{{\sum}}\limits_{l_{r}=0}^{N-r+1}
\left(\sum_{k=N+1}^\infty k^{s\mu}\, |\langle f, T_{k} \rangle |^s\right)^{2/s}
\left(\sum_{k=N+1}^\infty k^{(2r-\mu-1)\frac{s}{s-1}}\right)^{2(s-1)/s}
$$
$$
+ c\, \mathop{{\sum}}\limits_{l_{r}=N-r+2}^{\infty}
\left(\sum_{k=l_r+r}^\infty k^{s\mu}\, |\langle f, T_{k} \rangle |^s\right)^{2/s}
\left(\sum_{k=l_r+r}^\infty \, k^{(2r-\mu-1)\frac{s}{s-1}}\right)^{2(s-1)/s}
$$
$$
\leq c \|f\|_{s,\mu}^2 N^{-2(\mu-2r+1)+\frac{2(s-1)}{s}+1}.
$$
Thus, we find
$$
\|f^{(r)}-\mathcal{D}^{(r)}_N f\|_{2,\omega} \leq c \|f\|_{s,\mu} N^{-\mu+2r-1/s+1/2} .
$$

In the case of $s=1$ the assertion of Lemma is proved similarly. \vspace{0.1in}
\end{proof}

The following statement contains an estimate for the second difference from the right-hand side of (\ref{fullError}) in
the metric of $L_{2,\omega}$.

\begin{lemma}\label{lemma_BoundPertHC}
Let the condition (\ref{perturbation2}) be satisfied for $1\le p\le \infty$.
Then for arbitrary function $f\in L_{2,\omega}$ it holds
	$$
	\|\mathcal{D}^{(r)}_N f - \mathcal{D}^{(r)}_N f^\delta\|_{2,\omega} \leq c\delta N^{2r-1/p+1/2} .
	$$
\end{lemma}

\begin{proof}
We take the presentation (\ref{second_dif}) for the second difference of (\ref{fullError})
$$
\mathcal{D}^{(r)}_N f(t) - \mathcal{D}^{(r)}_N f^\delta(t)
= 2^r \sum_{k=r}^N k\, \langle f-f^\delta, T_{k} \rangle
\mathop{{\sum}}\limits_{l_1=r-1}^{k-1} l_1 \mathop{{\sum}}\limits_{l_2=r-2}^{l_1-1} l_2 \cdots
\mathop{{\sum}}\limits_{l_r=0}^{l_{r-1}-1} \xi_{l_r} T_{l_r}(t)
$$
and change the order of summation in it.
Recall that in this representation only those terms  take place for which all  indexes
$l_1+k, l_2+l_1,...,l_r+l_{r-1}$ are odd.
%Here we omit the symbol "*"\, , while taking into account this rule in the calculations.

Using the formula (\ref{dif_ort}), we have
$$
\mathcal{D}^{(r)}_N f(t) - \mathcal{D}^{(r)}_N f^\delta(t)
= 2^r \mathop{{\sum}}\limits_{l_r=0}^{N-r} \xi_{l_r} T_{l_r}(t)
\, \sum_{k=l_r+r}^N k\, \langle f-f^\delta, T_{k} \rangle B^r_k .
$$
At first, we consider the case $1<p<\infty.$
Using H\"{o}lder inequality, we get
$$
\|\mathcal{D}^{(r)}_N f - \mathcal{D}^{(r)}_N f^\delta\|_{2,\omega}^2
\le \frac{2^{2r+1}}{\pi} \mathop{{\sum}}\limits_{l_r=0}^{N-r}
\, \left(\sum_{k=l_r+r}^N k\, \langle f-f^\delta, T_{k} \rangle B^r_k\right)^2
$$
$$
\leq c\, \mathop{{\sum}}\limits_{l_r=0}^{N-r}
\, \left(\sum_{k=l_r+r}^N |\xi_k|^p\right)^{2/p}
\, \left(\sum_{k=l_r+r}^N k^{(2r-1)\frac{p}{p-1}}\right)^{2(p-1)/p}
$$
$$
\leq c\, \delta^2 N^{2(2r-1)+2(p-1)/p+1} ,
$$
which was required to prove.

In the cases of $p=1$ and $p=\infty$, the assertion of Lemma is proved similarly. \vspace{0.1in}
\end{proof}

\vskip 2mm

The combination of Lemmas \ref{lemma_BoundErrHC} and \ref{lemma_BoundPertHC} gives
\begin{theorem} \label{Th1}
	Let $f\in W^\mu_{s}$, $1\leq s< \infty$, $\mu>2r-1/s+1/2$, and let the condition (\ref{perturbation2}) be satisfied
	for $1\le p\le \infty$.  Then for $N\asymp \delta^{-\frac{1}{\mu-1/p+1/s}}$ it holds
	$$
	\|f^{(r)} - \mathcal{D}^{(r)}_N f^\delta\|_{2,\omega} \leq c \delta^{\frac{\mu-2r+1/s-1/2}{\mu-1/p+1/s}} .
	$$
\end{theorem}

\vskip 2mm

\begin{corollary} \label{Cor1}
In the considered problem, the truncation method  $\mathcal{D}^{(r)}_{N}$ (\ref{ModVer})
achieves the accuracy
	$$
	O\Big(\delta^{\frac{\mu-2r+1/s-1/2}{\mu-1/p+1/s}}\Big)
	$$
	on the class $W^{\mu}_{s}$, $\mu>2r-1/s+1/2$, and requires
	$$
	\card([r,N]) \asymp
	N \asymp \delta^{-\frac{1}{\mu-1/p+1/s}}
	$$
perturbed Fourier-Chebyshev coefficients.
\end{corollary}

\vskip 2mm

\section{Optimal recovery and information complexity in uniform metric}  \label{opt_C}

\subsection{General case of error in input data}

First, we will consider the case when, instead of the exact values of the Fourier-Chebyshev coefficients
of a differentiable function $f$, their perturbations $\overline{f^\delta}$ are known,
for which the estimate (\ref{perturbation2}) holds.

\begin{theorem} \label{Th4.1}
Let $1\leq s< \infty$, $1\leq p \leq \infty$, $\mu>2r-1/s+1$. Then for any $\delta>0$ and $N$
it holds
\begin{equation}  \label{low_est_p}
R_{N,\delta}^{(r)}(W^{\mu}_{s}, \Psi_N, C, \ell_p)
\geq {\cal E}_{\delta}^{(r)}(W^{\mu}_{s}, \Psi, C, \ell_p)
\geq c_p \, \delta^{\frac{\mu-2r+1/s-1}{\mu-1/p+1/s}} ,
\end{equation}
where
\begin{equation}  \label{c_p}
c_p = 3^{-\frac{\mu(2r-1/p+1)}{\mu-1/p+1/s}} \frac{\sqrt{2}}{\sqrt{\pi} (2r-1)!!} .
\end{equation}
\end{theorem}

\begin{proof}
For an arbitrary $\delta>0$, we choose $N_1$ so that
\begin{equation}  \label{N_1}
3^{-\mu} N_1^{-\mu+1/p-1/s} = \delta
\end{equation}
and construct an auxiliary function
\begin{equation}  \label{f_1}
f_1(t)  = 3^{-\mu} \,  N_1^{-\mu-1/s}\, \mathop{{\sum}}\limits_{k=N_1+r}^{2N_1+r-1}
T_k(t) .
\end{equation}
It is easy to see that $\|f_1\|_{s,\mu}\leq 1$. Further, since
$$
\|\overline{f_1}\|_{\ell_p}^p
= 3^{-\mu p} \,  N_1^{-\mu p - p/s + 1} ,
$$
then by the choice of $N_1$ (\ref{N_1}), we have
\begin{equation}  \label{f_1=delta}
\|\overline{f_1}\|_{\ell_p}
= 3^{-\mu} \,  N_1^{-\mu +1/p - 1/s} = \delta.
\end{equation}
Let us estimate the value $\|f_1^{(r)}\|_{C}$ from below.
By (\ref{dif_ort}) we write down the representation
$$
f_1^{(r)}(t)  = 3^{-\mu} \,  N_1^{-\mu-1/s}\, \mathop{{\sum}}\limits_{k=N_1+r}^{2N_1+r-1}
T_k^{(r)}(t)
$$
\begin{equation}  \label{RHS}
= 2^r\, 3^{-\mu} \,  N_1^{-\mu-1/s}\, \mathop{{\sum}}\limits_{k=N_1+r}^{2N_1+r-1} k
\mathop{{\sum}}\limits_{l_1=r-1}^{k-1} l_1 \mathop{{\sum}}\limits_{l_2=r-2}^{l_1-1} l_2 \ldots
\mathop{{\sum}}\limits_{l_{r-1}=1}^{l_{r-2}-1} l_{r-1}
\mathop{{\sum}}\limits_{l_{r}=0}^{l_{r-1}-1} \xi_{l_r} T_{l_r}(t) .
\end{equation}
We note that in the right-hand side of (\ref{RHS}) only terms with odd indexes $l_1+k, l_2+l_1,..., l_r+l_{r-1}$ take part.
It is easy to see that
$$
\|f_1^{(r)}\|_{C}  \ge  |f_1^{(r)}(1)|
= \frac{2^{r+1/2}}{\sqrt{\pi}}\, 3^{-\mu} \,  N_1^{-\mu-1/s}\, \mathop{{\sum}}\limits_{k=N_1+r}^{2N_1+r-1} k
$$
$$
\times \mathop{{\sum}}\limits_{l_1=r-1}^{k-1} l_1 \mathop{{\sum}}\limits_{l_2=r-2}^{l_1-1} l_2 \ldots
\mathop{{\sum}}\limits_{l_{r-1}=1}^{l_{r-2}-1} l_{r-1}
\mathop{{\sum}}\limits_{l_{r}=0}^{l_{r-1}-1} 1
$$
$$
\ge \frac{\sqrt{2} N_1^{-\mu+2r-1/s+1}}{3^{\mu} \sqrt{\pi} (2r-1)!!} .
$$
This means
\begin{equation}  \label{f_1(r)}
\|f_1^{(r)}\|_{C}
\geq c_p \, \delta^{\frac{\mu-2r+1/s-1}{\mu-1/p+1/s}} .
\end{equation}

Next, we use a well-known scheme for obtaining lower bounds (see, for example, \cite{Os}).
Since $f_1, -f_1\in W^\mu_s$, then according to (\ref{f_1=delta}), we have that
$\overline{0} =(0,\ldots,0,\ldots)\in \ell_p$ is a $\delta$-perturbation of both
$\overline{f_1}$ and $-\overline{f_1}$.
Therefore, for any $\psi^{(r)} G\in \Psi$ it is true
$$
2\|f_1^{(r)}\|_{C}
\le \|f_1^{(r)} - \psi^{(r)}G(0)\|_{C}
+ \|-f_1^{(r)} - \psi^{(r)}G(0)\|_{C} .
$$
This implies
$$
\|f_1^{(r)}\|_{C}
\leq \sup_{\substack{\|f\|_{s,\mu}\leq 1},
\atop
% \ \sup_{\substack{\overline{f^\delta}: \, \overline{f}^{\delta}=\overline{f}+\overline{\xi}}
% \atop
\|\overline{f}\|_{\ell_p} \leq \delta }
\| f^{(r)}\|_{C}
 \leq \sup_{\substack{\|f\|_{s,\mu}\leq 1}}
 \ \sup_{\substack{f^\delta\in L_{2,\omega}: \, \overline{f}^{\delta}=\overline{f}+\overline{\xi}}
 \atop \|\overline{\xi}\|_{\ell_p} \leq \delta }
 \| f^{(r)} - \psi^{(r)}G(f^{\delta}) \|_{C} .
$$
Due to the arbitrariness of $\psi^{(r)} G\in \Psi$, using (\ref{f_1(r)}) we get
\begin{equation}  \label{E_delta}
{\cal E}^{(r)}_{\delta}(W^{\mu}_{s}, \Psi, C, \ell_p)
\geq \| f_1^{(r)}\|_{C}
\geq c_p \, \delta^{\frac{\mu-2r+1/s-1}{\mu-1/p+1/s}} .
\end{equation}
Substituting (\ref{E_delta}) into (\ref{V}), we obtain the statement of Theorem.
\end{proof}

The following statement contains order-optimal estimates for the quantities
${\cal E}^{(r)}_{\delta}$, $R_{N,\delta}^{(r)}$ in the uniform metric.

{\begin{theorem} \label{Th4.2}
Let $1\leq s< \infty$, $\mu>2r-1/s+1$, and let the condition (\ref{perturbation2})
be satisfied for $1\le p\le \infty$.
Then for any $\delta>0$ it holds
$$
{\cal E}_{\delta}^{(r)}(W^{\mu}_{s},\Psi, C, \ell_p)
\asymp \delta^{\frac{\mu-2r+1/s-1}{\mu-1/p+1/s}} .
$$
Moreover, for any $\delta>0$ and $N \asymp \delta^{-\frac{1}{\mu-1/p+1/s}}$ it holds
$$
R_{N,\delta}^{(r)}(W^{\mu}_{s},\Psi_N, C, \ell_p)
\asymp \delta^{\frac{\mu-2r+1/s-1}{\mu-1/p+1/s}}
\asymp N^{-\mu+2r-1/s+1} .
$$
The order-optimal bounds of ${\cal E}^{(r)}_{\delta}$, $R_{N,\delta}^{(r)}$
are implemented by the method (\ref{ModVer}).
\end{theorem}}

\vskip 2mm

The assertion of Theorem \ref{Th4.2} follows immediately from Theorem \ref{Th2} (upper bound) and
Theorem \ref{Th4.1} (lower bound).

\vskip 4mm

\begin{remark} \label{Legendre}
Comparison of Theorem \ref{Th4.2} with the results of \cite{Lu&Naum&Per} and \cite{Sol_Stas_UMZ2022} shows that, in the $C$-metric, Chebyshev polynomials provide a higher accuracy of numerical differentiation compared to Legendre polynomials. For example, for $p=s=2$, the highest order of recovery of the first derivative $f^{(1)}$ by Chebyshev polynomials is $O(\delta^{\frac{\mu-5/2}{\mu}})$ and by Legendre polynomials is $O(\delta^\frac{\mu-3}{\mu})$.
\end{remark}	
	
\vskip 2mm

\begin{remark} \label{change1}
As follows from Theorem \ref{Th4.2}, narrowing the set of methods from $\Psi$
to $\Psi_N$ does not affect the optimal order of accuracy $O(\delta^{\frac{\mu-2r+1/s-1}{\mu-1/p+1/s}})$.
In other words, no method from $\Psi$ dealing with an arbitrary (possibly infinite)
amount of Galerkin information provides a higher order  accuracy than method (\ref{ModVer}).
\end{remark}

\vskip 2mm
		
{It remains to establish the smallest value of $N$ at which this order is achieved for the quantity
$R_{N,\delta}^{(r)}$ in the uniform metric.}

{\begin{theorem} \label{Th4.3}
Let $1\leq s< \infty$, $1\leq p \leq \infty$, $\mu>2r-1/s+1$.
Then for any $\delta>0$ and $N$ it holds
\begin{equation}  \label{comb}
R_{N,\delta}^{(r)}(W^{\mu}_{s},\Psi_N, C, \ell_p)
\ge c\, \max \left\{
\delta^{\frac{\mu-2r+1/s-1}{\mu-1/p+1/s}},\, \,
N^{-\mu+2r-1/s+1}\right\} .
\end{equation}
\end{theorem}}

\begin{proof}
{Let $\delta$ and $N$ be arbitrary.
Let us fix an arbitrarily chosen set $\hat{\Omega}$, $\card(\hat{\Omega})\le N$, of points
$k$, $r\le k < \infty$, of the coordinate axis and take a set
$$
\Lambda_N = \{k: N+r\le k\le 3N+r, \, k \notin \hat{\Omega}\} ,
$$
consisting of $N$ points of the axis such that
$$
\hat{\Omega} \cap \Lambda_N = \varnothing .
$$
Note that such a set $\Lambda_N$ will always exist.
Further, we construct the following auxiliary function
$$
f_2 (t) = 4^{-\mu}\, N^{-\mu-1/s}\, \sum_{k\in \Lambda_N} T_k (t)
$$
and estimate the norm of $f_2$ in the metric of $W^{\mu}_{s}$:}
$$
\|f_2\|_{s,\mu}^s = 4^{-\mu s}\, N^{-\mu s - 1}\, \sum_{k\in \Lambda_N} k^{\mu s} \le 1 .
$$
Using (\ref{dif_ort}), we write down the representation
$$
f_2^{(r)} (t)
= 2^r\, 4^{-\mu} \,  N^{-\mu-1/s}\, \mathop{{\sum}}\limits_{k\in \Lambda_N} k
\mathop{{\sum}}\limits_{l_1=r-1}^{k-1} l_1 \mathop{{\sum}}\limits_{l_2=r-2}^{l_1-1} l_2 \ldots
\mathop{{\sum}}\limits_{l_{r-1}=1}^{l_{r-2}-1} l_{r-1}
\mathop{{\sum}}\limits_{l_{r}=0}^{l_{r-1}-1} \xi_{l_r} T_{l_r}(t) .
$$
Then
$$
\|f_2^{(r)}\|_{C}  \ge  |f_2^{(r)}(1)|
= \frac{2^{r+1/2}}{\sqrt{\pi}}\, 4^{-\mu} \,  N^{-\mu-1/s}\, \mathop{{\sum}}\limits_{k\in \Lambda_N} k
$$
$$
\times \mathop{{\sum}}\limits_{l_1=r-1}^{k-1} l_1 \mathop{{\sum}}\limits_{l_2=r-2}^{l_1-1} l_2 \ldots
\mathop{{\sum}}\limits_{l_{r-1}=1}^{l_{r-2}-1} l_{r-1}
\mathop{{\sum}}\limits_{l_{r}=0}^{l_{r-1}-1} 1
$$
\begin{equation}  \label{f_2(r)}
\ge c' N^{-\mu+2r-1/s+1} ,
\end{equation}
where
$$
c'= \frac{\sqrt{2}}{4^{\mu} \sqrt{\pi} (2r-1)!!} .
$$
It is obvious that for any $\delta>0$, under a $\delta$-perturbation $f_2^\delta$ of the function $f_2$,
we can consider the function $f_2$ itself, i.e. $f_2^\delta (t) := f_2 (t)$.
Taking into account the relationship
$G_{\hat{\Omega}}(f_2)=G_{\hat{\Omega}}(-f_2)=\overline{0}$,
for any $\psi^{(r)}G_{\hat{\Omega}}\in\Psi_{\hat{\Omega}}$ we have
$\psi^{(r)}G_{\hat{\Omega}}(f_2)=\psi^{(r)}G_{\hat{\Omega}}(-f_2)$ and
$$
2 \|f_2^{(r)}\|_{C}
=  \|-f_2^{(r)} - \psi^{(r)}G_{\hat{\Omega}}(-f_2) - f_2^{(r)} + \psi^{(r)}G_{\hat{\Omega}}(f_2)\|_{C}
$$
$$
\le \|- f_2^{(r)} - \psi^{(r)}G_{\hat{\Omega}}(-f_2^\delta)\|_{C}
+ \|f_2^{(r)} - \psi^{(r)}G_{\hat{\Omega}}(f_2^\delta)\|_{C}
$$
$$
\le 2\, \sup_{\substack{\|f\|_{s,\mu}\leq 1}}
 \ \sup_{\substack{f^\delta\in L_{2,\omega}: \, \overline{f}^{\delta}=\overline{f}+\overline{\xi}}
 \atop \|\overline{\xi}\|_{\ell_p} \leq \delta }
 \| f^{(r)} - \psi^{(r)}G_{\hat{\Omega}}(f^{\delta}) \|_{C}
$$
$$
:= 2\, e_\delta (W_s^\mu, \psi^{(r)}G_{\hat{\Omega}}, C, \ell_p) .
$$
Then, using (\ref{f_2(r)}), we find that for any $\delta>0$ and $N$, the following holds:
$$
e_\delta (W_s^\mu, \psi^{(r)}G_{\hat{\Omega}}, C, \ell_p)
\ge \|f_2^{(r)}\|_{C}
\ge c'\, N^{-\mu+2r-1/s+1} .
$$
From here, due to arbitrariness of $\psi^{(r)}G_{\hat{\Omega}}$ and $\hat{\Omega}$, we get
\begin{equation}  \label{low_est_N}
R_{N,\delta}^{(r)}(W^{\mu}_{s},\Psi_N, C, \ell_p)
\ge  c'\, N^{-\mu+2r-1/s+1} .
\end{equation}
Combining (\ref{low_est_p}) and (\ref{low_est_N}), we obtain the statement of Theorem.
\end{proof}

\begin{corollary} \label{Cor3}
From (\ref{comb}) it follows that for any $N\ll \delta^{-1/(\mu+1/s-1/p)}$ the relation
$$
R_{N,\delta}^{(r)}(W^{\mu}_{s},\Psi_N, C, \ell_p)
\gg \delta^{\frac{\mu-2r+1/s-1}{\mu-1/p+1/s}}
$$
is satisfied. This means that for such $N$ the optimal order of accuracy cannot be achieved.
On the other hand, for all $N\ge c\, \delta^{-1/(\mu+1/s-1/p)}$ the optimal order can be achieved.
\end{corollary}

Thus, we get
\begin{theorem} \label{ThNEW}
Let $1\leq s< \infty$, $1\leq p \leq \infty$, $\mu>2r-1/s+1$.
The most economical (in order) choice of $N$ gives the rule
\begin{equation}  \label{optNp}
N_{\min}\asymp \delta^{-1/(\mu+1/s-1/p)} ,
\end{equation}
for which it holds
%In other words, the smallest (in order) value of $N$, for which the relation
$$
\min\limits_N R_{N,\delta}^{(r)}(W^{\mu}_{s},\Psi_N, C, \ell_p)
\asymp R_{N_{\min},\delta}^{(r)}(W^{\mu}_{s},\Psi_{N_{\min}}, C, \ell_p)
\asymp \delta^{\frac{\mu-2r+1/s-1}{\mu-1/p+1/s}} .
$$
%is true, satisfies (\ref{optNp}).
\end{theorem}

\begin{remark}  \label{rmNEW}
In the framework of  IBC theory, the quantity $N_{\min}$ is usually called information complexity.
The value (\ref{optNp}) describes the smallest (in order) amount of discrete information of the form (\ref{FHc})
necessary to achieve the highest possible accuracy $O(\delta^{\frac{\mu-2r+1/s-1}{\mu-1/p+1/s}})$.
\end{remark}

\subsection{Case of $L_{2,\omega}$-error in input data}

Now we consider separately the case where the error of the input data is measured in $L_{2,\omega}$, i.e.
$\|f-f^\delta\|_{2,\omega} \leq \delta$. In this case, it is possible to establish sharp (in order) estimates for
$E_\delta^{(r)}$ as well.

\begin{theorem} \label{Th4.4}
Let $1\leq s< \infty$, $\mu>2r-1/s+1$. Then for any $\delta>0$
%and $N$
it holds
%$$
%R_{N,\delta}^{(r)}(W^{\mu}_{s}, \Psi_N, L_2, \ell_2)
%\geq E_{\delta}^{(r)}(W^{\mu}_{s}, \Psi, L_2, \ell_2)
%$$
\begin{equation}  \label{low_est_2}
E_{\delta}^{(r)}(W^{\mu}_{s}, {\cal M}, C)
\geq c_2 \, \delta^{\frac{\mu-2r+1/s-1}{\mu+1/s-1/2}} ,
\end{equation}
where $c_2$ is defined by (\ref{c_p}).
\end{theorem}

\begin{proof}
To obtain the lower bound (\ref{low_est_2}) for $E_{\delta}^{(r)}$, we use the auxiliary function $f_1(t)$ (\ref{f_1}).
Now we recall the known properties of $f_1$ (see Theorem \ref{Th4.1}):
\begin{equation}  \label{prop_f_1}
\|f_1\|_{s,\mu} \le 1, \qquad
\|f_1\|_{2,\omega} = \|\overline{f_1}\|_{\ell_2} = \delta, \qquad
\|f_1^{(r)}\|_{C}
\geq c_2 \, \delta^{\frac{\mu-2r+1/s-1}{\mu+1/s-1/2}} .
\end{equation}
Next, repeating the reasoning from the proof of Theorem \ref{Th4.1}, we obtain
that for any solution operator $m^{(r)} \in {\mathcal M}$ the following holds:
$$
2\|f_1^{(r)}\|_{C}
\le \|f_1^{(r)} - m^{(r)}0\|_{C}
+ \|-f_1^{(r)} - m^{(r)}0\|_{C} .
$$
This implies
$$
\|f_1^{(r)}\|_{C} \le
\sup_{\substack{\|f\|_{s,\mu}\leq 1},
\atop
% \ \sup_{\substack{\overline{f^\delta}: \, \overline{f}^{\delta}=\overline{f}+\overline{\xi}}
% \atop
\|f\|_{2,\omega} \leq \delta }
\| f^{(r)}\|_{C}
 \leq \sup_{\substack{\|f\|_{s,\mu}\leq 1}}
 \ \sup_{\substack{f^\delta\in L_{2,\omega},
 \atop \|f-f^\delta\|_{2,\omega} \leq \delta }}
 \| f^{(r)} - m^{(r)}f^{\delta} \|_{C} .
$$
Due to the arbitrariness of $m^{(r)} \in {\mathcal M}$, using (\ref{prop_f_1}) we get the desired estimate.
%\begin{equation}  \label{epsilon}
%\varepsilon^{(r)}_{\delta}(W^{\mu}_{s}, {\mathcal M}, L_2)
%%\geq \| f_1^{(r)}\|_{2}
%\geq c_2 \, \delta^{\frac{\mu-r/2+1/s-1/2}{\mu+1/s-1/2}} .
%\end{equation}
%Substituting (\ref{epsilon}) into (\ref{IV}), we obtain the statement of Theorem.
\end{proof}

The combination of Theorem \ref{Th2} (upper bound for $p=2$), Theorem \ref{Th4.4} (lower bound) and relation (\ref{IV})
yields the following statement.

\begin{theorem} \label{Th4.5}
Let $1\leq s< \infty$, $\mu>2r-1/s+1$.
Then for any $\delta>0$ it holds
$$
E^{(r)}_{\delta}(W^{\mu}_{s}, {\mathcal M}, C)
\asymp {\cal E}_{\delta}^{(r)}(W^{\mu}_{s},\Psi, C, \ell_2)
\asymp \delta^{\frac{\mu-2r+1/s-1}{\mu+1/s-1/2}} .
$$
Moreover, for any $\delta>0$ and $N \asymp \delta^{-\frac{1}{\mu+1/s-1/2}}$ it holds
$$
R_{N,\delta}^{(r)}(W^{\mu}_{s},\Psi_N, C, \ell_2)
\asymp \delta^{\frac{\mu-2r+1/s-1}{\mu+1/s-1/2}}
\asymp N^{-\mu+2r-1/s+1} .
$$
The order-optimal bounds of $E^{(r)}_{\delta}$, ${\cal E}^{(r)}_{\delta}$, $R_{N,\delta}^{(r)}$
are implemented by the method (\ref{ModVer}).
\end{theorem}

\vskip 4mm

{\begin{remark} \label{change2}
As follows from Theorem \ref{Th4.5}, narrowing the set of methods from ${\mathcal M}$
to $\Psi_N$ and $\Psi$ does not affect the optimal order of accuracy $O(\delta^{\frac{\mu-2r+1/s-1}{\mu+1/s-1/2}})$.
In other words, no mapping from $L_{2,\omega}$ into $C$ provides a higher order of accuracy than the method (\ref{ModVer}).
%From Corollary \ref{Cor3} it follows that rule
%\begin{equation}  \label{optN2}
%N\asymp \delta^{-1/(\mu+1/s-1/2)} ,
%\end{equation}
%is the most economical (in order) choice for $N$, i.e.
%the smallest (in order) value of $N$, for which the relation
%$$
%\inf\limits_N R_{N,\delta}^{(r)}(W^{\mu}_{s},\Psi_N, C, \ell_2)
%\asymp\delta^{\frac{\mu-2r+1/s-1}{\mu-1/2+1/s}}
%$$
%is true, satisfies (\ref{optN2}).
\end{remark}}

\section{Optimal recovery and information complexity in the integral metric} \label{opt_L2}

{\begin{theorem} \label{Th5.1}
Let $1\leq s< \infty$, $\mu>2r-1/s+1/2$, $1\leq p \leq \infty$. Then for any
$\delta>0$ and $N$ it holds
$$
%\begin{equation}  \label{low_est_p}
R_{N,\delta}^{(r)}(W^{\mu}_{s}, \Psi_N, L_{2,\omega}, \ell_p)
\geq {\cal E}_{\delta}^{(r)}(W^{\mu}_{s}, \Psi, L_{2,\omega}, \ell_p)
\geq \overline{c}_p \, \delta^{\frac{\mu-2r+1/s-1/2}{\mu-1/p+1/s}} ,
%\end{equation}
$$
where
\begin{equation}  \label{ov_c_p}
\overline{c}_p = 3^{-\frac{\mu(2r-1/p+1/2)}{\mu-1/p+1/s}} c^* .
\end{equation}
\end{theorem}}

\begin{proof}
The proof of Theorem \ref{Th5.1} almost completely coincides with the proof of Theorem \ref{Th4.1},
including the form of the auxiliary functions
$f_1$, $f_1^{\delta}=0$.
The only difference is in the lower estimate of the norm of
$f_1^{(r)}$. Changing the order of summation in (\ref{RHS}) yields to the representation
	$$
	f_1^{(r)}(t) = 2^r\, 3^{-\mu} \, {N_1^{-\mu-1/s}} \,
	\Big(\mathop{{\sum}}\limits_{l_r=0}^{N_1} \xi_{l_r} T_{l_r}(t) \mathop{{\sum}
	}\limits_{k=N_1+r}^{2N_1+r-1}\, k
	+ \mathop{{\sum}}\limits_{l_r=N_1+1}^{2N_1-1} \xi_{l_r} T_{l_r}(t) \mathop{{\sum}}\limits_{k=l_r+r}^{2N_1+r-1}\, k\Big)
	B^r_k .
	$$
Taking into account (\ref{Br}), it is easy to verify that
	$$
	\|f_1^{(r)}\|_{2,\omega}^2
	\geq  2^{2r}\,  3^{-2\mu} \, (c^*)^2 N^{-2\mu+4r-2/s+1} .
	$$
Next, repeating the reasoning from Theorem \ref{Th4.1}, we obtain for any $\delta>0$
$$
{\cal E}_{\delta}^{(r)}(W^{\mu}_{s}, \Psi, L_{2,\omega}, \ell_p)
\geq \overline{c}_p \, \delta^{\frac{\mu-2r+1/s-1/2}{\mu-1/p+1/s}} .
$$
By using (\ref{V}) we get the assertion of Theorem.
\end{proof}

\vskip 2mm

The following statement contains order-optimal estimates for the quantities
${\cal E}^{(r)}_{\delta}$, $R_{N,\delta}^{(r)}$ in the integral metric.

\begin{theorem} \label{Th5.2}
Let $1\leq s< \infty$, $\mu>2r-1/s+1$, and let the condition (\ref{perturbation2})
be satisfied for $1\le p\le \infty$.
Then for any $\delta>0$ it holds
$$
{\cal E}_{\delta}^{(r)}(W^{\mu}_{s},\Psi, L_{2,\omega}, \ell_p)
\asymp \delta^{\frac{\mu-2r+1/s-1/2}{\mu-1/p+1/s}} .
$$
Moreover, for any $\delta>0$ and $N \asymp \delta^{-\frac{1}{\mu-1/p+1/s}}$ it holds
$$
R_{N,\delta}^{(r)}(W^{\mu}_{s},\Psi_N, L_{2,\omega}, \ell_p)
\asymp \delta^{\frac{\mu-2r+1/s-1/2}{\mu-1/p+1/s}}
\asymp N^{-\mu+2r-1/s+1/2} .
$$
The order-optimal bounds of ${\cal E}^{(r)}_{\delta}$, $R_{N,\delta}^{(r)}$
are implemented by the method (\ref{ModVer}).
\end{theorem}

\vskip 2mm

The assertion of Theorem \ref{Th5.2} follows immediately from Theorem \ref{Th1} (upper bound) and
Theorem \ref{Th5.1} (lower bound).
	
\vskip 2mm

\begin{remark} \label{2change1}
As follows from Theorem \ref{Th5.2}, narrowing the set of methods from $\Psi$
to $\Psi_N$ does not affect the optimal order of accuracy $O(\delta^{\frac{\mu-2r+1/s-1/2}{\mu-1/p+1/s}})$.
In other words, no differentiation method from $\Psi$ dealing with an arbitrary (possibly infinite)
amount of Galerkin information provides a higher order of accuracy than method (\ref{ModVer}).
\end{remark}

\vskip 2mm

Further, we establish the smallest value of $N$ at which one can achieve
the sharp (in the power scale) estimates for the quantity
%we establish the smallest value of $N$ at which this order is achieved for the quantity
$R_{N,\delta}^{(r)}$ in the integral metric.

\begin{theorem} \label{Th5.3}
Let $1\leq s< \infty$, $1\leq p \leq \infty$, $\mu>2r-1/s+1/2$.
Then for any $\delta>0$ and $N$ it holds
\begin{equation}  \label{combL}
R_{N,\delta}^{(r)}(W^{\mu}_{s},\Psi_N, L_{2,\omega}, \ell_p)
\ge c\, \max \left\{
\delta^{\frac{\mu-2r+1/s-1/2}{\mu-1/p+1/s}},\,\,
N^{-\mu+2r-1/s+1/2}\right\} .
\end{equation}
\end{theorem}

The proof of Theorem \ref{Th5.3} is similar to the proof of Theorem \ref{Th4.3}.

\vskip 2mm

\begin{corollary} \label{Cor4}
From (\ref{combL}) it follows that for any $N\ll \delta^{-1/(\mu+1/s-1/p)}$ the relation
$$
R_{N,\delta}^{(r)}(W^{\mu}_{s},\Psi_N, L_{2,\omega}, \ell_p)
\gg \delta^{\frac{\mu-2r+1/s-1/2}{\mu-1/p+1/s}}
$$
is satisfied. This means that for such $N$ the optimal order of accuracy cannot be achieved.
On the other hand, for all $N\ge c\, \delta^{-1/(\mu+1/s-1/p)}$ he optimal accuracy is achieved.
\end{corollary}

Thus, we get

{\begin{theorem} \label{ThNEW1}
Let $1\leq s< \infty$, $1\leq p \leq \infty$, $\mu>2r-1/s+1/2$.
The most economical (in order) choice of $N_{\min}$ gives the rule (\ref{optNp}),
for which it holds
$$
\min\limits_N R_{N,\delta}^{(r)}(W^{\mu}_{s},\Psi_N, L_{2,\omega}, \ell_p)
\asymp R_{N_{\min},\delta}^{(r)}(W^{\mu}_{s},\Psi_{N_{\min}}, L_{2,\omega}, \ell_p)
\asymp\delta^{\frac{\mu-2r+1/s-1/2}{\mu-1/p+1/s}} .
$$
%is true, satisfies (\ref{optNp}).
\end{theorem}}

\vskip 2mm

Now we consider separately the case where the error of the input data is measured in $L_{2,\omega}$.

\begin{theorem} \label{Th5.4}
Let $1\leq s< \infty$, $\mu>2r-1/s+1/2$. Then for any $\delta>0$
%and $N$
it holds
%\begin{equation}  \label{low_est_2}
$$
E_{\delta}^{(r)}(W^{\mu}_{s}, {\cal M}, L_{2,\omega})
\geq \overline{c}_2 \, \delta^{\frac{\mu-2r+1/s-1/2}{\mu+1/s-1/2}} ,
$$
%\end{equation}
where $\overline{c}_2$ is defined by (\ref{ov_c_p}).
\end{theorem}

The technique of the proof of Theorem \ref{Th5.4} shares many of the same principles as Theorem \ref{Th4.4}.

The combination of Theorem \ref{Th1} (upper bound for $p=2$), Theorem \ref{Th5.4} (lower bound) and relation (\ref{IV})
yields the following statement.

\begin{theorem} \label{Th5.5}
Let $1\leq s< \infty$, $\mu>2r-1/s+1/2$.
Then for any $\delta>0$ it holds
$$
E^{(r)}_{\delta}(W^{\mu}_{s}, {\mathcal M}, L_{2,\omega})
\asymp {\cal E}_{\delta}^{(r)}(W^{\mu}_{s},\Psi, L_{2,\omega}, \ell_2)
\asymp \delta^{\frac{\mu-2r+1/s-1/2}{\mu+1/s-1/2}} .
$$
Moreover, for any $\delta>0$ and $N \asymp \delta^{-\frac{1}{\mu+1/s-1/2}}$ it holds
$$
R_{N,\delta}^{(r)}(W^{\mu}_{s},\Psi_N, L_{2,\omega}, \ell_2)
\asymp \delta^{\frac{\mu-2r+1/s-1/2}{\mu+1/s-1/2}}
\asymp N^{-\mu+2r-1/s+1/2} .
$$
The order-optimal bounds of $E^{(r)}_{\delta}$, ${\cal E}^{(r)}_{\delta}$, $R_{N,\delta}^{(r)}$
are implemented by the method (\ref{ModVer}).
\end{theorem}

\vskip 2mm

\begin{remark} \label{2change2}
As follows from Theorem \ref{Th5.5}, narrowing the set of methods from ${\mathcal M}$
to $\Psi_N$ and $\Psi$ does not affect the optimal order of accuracy $O(\delta^{\frac{\mu-2r+1/s-1/2}{\mu+1/s-1/2}})$.
In other words, no mapping from $L_{2,\omega}$ into $L_{2,\omega}$ provides a higher order of accuracy than the method (\ref{ModVer}).
%From Corollary \ref{Cor4} it follows that rule (\ref{optN2})
%is the most economical (in order) choice for $N$, i.e.
%the smallest (in order) value of $N$, for which the relation
%$$
%\inf\limits_N R_{N,\delta}^{(r)}(W^{\mu}_{s},\Psi_N, L_{2,\omega}, \ell_2)
%\asymp\delta^{\frac{\mu-2r+1/s-1/2}{\mu-1/2+1/s}}
%$$
%is true, satisfies (\ref{optN2}).
\end{remark}

\section{Numerical summation}   \label{num_prob}
		
If we put $r=0$ in (\ref{r_deriv}) and in the definitions of $E^{(r)}_{\delta}$, ${\cal E}^{(r)}_{\delta}$, $R_{N,\delta}^{(r)}$, then the problem of numerical differentiation is transformed into a problem of numerical summation.
Our further research is devoted to estimating the quantities $E^{(0)}_{\delta}$, ${\cal E}^{(0)}_{\delta}$, $R_{N,\delta}^{(0)}$,
which characterizes the optimal recovery
and information complexity for the numerical summation problem.
It should be noted that numerical summation as an unstable problem was considered in \cite{Tikh1964},
where the problem of ensuring the stability of approximations for periodic functions with imprecisely specified
Fourier coefficients was first studied.
It is well known that the stability of this problem depends on the metrics of the spaces in which it is studied.
		Thus, in particular, it turned out (see, for example, \cite{TAr1977}) that for periodic functions in the case
		of measuring the error of the input data in $\ell_2$, the summation problem is well-posed if the accuracy
		of the approximation is estimated in $L_2$-norm, and unstable if the approximation accuracy is estimated in $C$.
		The results presented below not only confirm these conclusions for non-periodic functions, but also allow us to establish that
		(and how) the stability of the problem is affected by the values of the parameters $p$, $s$ and the metric of the space $X$.
		Thus, Theorems \ref{Sum_C}, \ref{2Sum_C}, \ref{Sum_L} and \ref{2Sum_L} make it possible to create a clear view of the conditions
		of stability and instability in the problem of numerical summation.
		%Moreover, in this study, we present for the first time order-optimal estimates of the minimal radius
		%in the case where Galerkin information is understood as perturbed values of the Fourier-Chebyshev coefficients
		%$\langle f^\delta, T_{k} \rangle$.
		
		For numerical summation of functions $f\in W^{\mu}_{s}$ representable as the series (\ref{function}),
		we use the following variant of truncation method
		\begin{equation} \label{ModVer1}
			\mathcal{D}_N^{(0)} f^\delta(t) = \sum_{k=0}^{N} \langle f^\delta, T_{k}\rangle
			T^{}_k(t) .
		\end{equation}
		
		\vskip 2mm
		
The following two statements contain order-optimal estimates of the quantities
$E^{(0)}_{\delta}$, ${\cal E}^{(0)}_{\delta}$, $R_{N,\delta}^{(0)}$ in the uniform metric.
		
\begin{theorem} \label{Sum_C}
Let $1\leq s< \infty$, $\mu>1-1/s$, $1\leq p \leq \infty$.
Then for any $\delta>0$ it holds
$$
{\cal E}_{\delta}^{(0)}(W^{\mu}_{s},\Psi, C, \ell_p)
\asymp \delta^{\frac{\mu+1/s-1}{\mu-1/p+1/s}} .
$$
Moreover, for any $\delta>0$ and $N \asymp \delta^{-\frac{1}{\mu-1/p+1/s}}$ it holds
$$
R_{N,\delta}^{(0)}(W^{\mu}_{s},\Psi_N, C, \ell_p)
\asymp \delta^{\frac{\mu+1/s-1}{\mu-1/p+1/s}}
\asymp N^{-\mu-1/s+1} .
$$
The order-optimal bounds are implemented by the method $\mathcal{D}_N^{(0)}$ (\ref{ModVer1}).
\end{theorem}
		
Theorem \ref{Sum_C} can be proven in the same way as Theorem \ref{Th4.2}.

\vskip 2mm

\begin{theorem} \label{2Sum_C}
Let $1\leq s< \infty$, $\mu>1-1/s$.
Then for any $\delta>0$ it holds
$$
{\cal E}_{\delta}^{(0)}(W^{\mu}_{s},\Psi, C, \ell_2)
\asymp E^{(0)}_{\delta}(W^{\mu}_{s}, {\mathcal M}, C)
\asymp \delta^{\frac{\mu+1/s-1}{\mu+1/s-1/2}} .
$$
Moreover, for any $\delta>0$ and $N \asymp \delta^{-\frac{1}{\mu+1/s-1/2}}$ it holds
$$
R_{N,\delta}^{(0)}(W^{\mu}_{s},\Psi_N, C, \ell_2)
\asymp \delta^{\frac{\mu+1/s-1}{\mu+1/s-1/2}}
\asymp N^{-\mu-1/s+1} .
$$
The order-optimal bounds are implemented by the method (\ref{ModVer1}).
\end{theorem}

Theorem \ref{2Sum_C} can be proven in the same way as Theorem \ref{Th4.5}.

\vskip 2mm

\begin{remark} \label{norma_C}
As follows from Theorem \ref{Sum_C},
for $p=1$ and any $1\le s< \infty$, the summation problem in the metric of $C$ is well-posed.
This means that for the indicated values of $p$ and $s$, the problem is also well-posed in any weaker metric,
for example, in $L_q$, $1\le q\le \infty$,
where by $L_{q}$ we mean, as usual, the space of real-valued functions $f(t)$ with the norm
$\|f\|_{q} ;= \left(\int_{-1}^1 |f(t)|^q\, dt\right)^{1/q} < \infty$.
But even in these cases, it is necessary to adjust the levels of discretization $N$ and perturbation $\delta$
to prevent the saturation effect.
\end{remark}
		
		\vskip 2mm
		
The following two statements contain order-optimal estimates of the quantities
$E^{(0)}_{\delta}$, ${\cal E}^{(0)}_{\delta}$, $R_{N,\delta}^{(0)}$ in the integral metric.
		
\begin{theorem} \label{Sum_L}
Let $2\leq s< \infty$, $\mu>1/2-1/s$, $2\leq p \leq \infty$.
Then for any $\delta>0$ it holds
$$
{\cal E}_{\delta}^{(0)}(W^{\mu}_{s},\Psi, L_{2,\omega}, \ell_p)
\asymp \delta^{\frac{\mu+1/s-1/2}{\mu-1/p+1/s}} .
$$
Moreover, for any $\delta>0$ and $N \asymp \delta^{-\frac{1}{\mu-1/p+1/s}}$ it holds
$$
R_{N,\delta}^{(0)}(W^{\mu}_{s},\Psi_N, L_{2,\omega}, \ell_p)
\asymp \delta^{\frac{\mu+1/s-1/2}{\mu-1/p+1/s}}
\asymp N^{-\mu-1/s+1/2} .
$$
The order-optimal bounds are implemented by the method $\mathcal{D}^{(0)}_{N}$ (\ref{ModVer1}).
\end{theorem}
		
Theorem \ref{Sum_L} can be proven in the same way as Theorem \ref{Th5.2}.
\vskip 2mm
		
\begin{theorem} \label{2Sum_L}
Let $2\leq s< \infty$, $\mu>1/2-1/s$.
Then for any $\delta>0$ it holds
$$
E^{(0)}_{\delta}(W^{\mu}_{s}, {\mathcal M}, L_{2,\omega})
\asymp {\cal E}_{\delta}^{(0)}(W^{\mu}_{s},\Psi, L_{2,\omega}, \ell_2)
\asymp \delta .
$$
Moreover, for any $\delta>0$ and $N \asymp \delta^{-\frac{1}{\mu+1/s-1/2}}$ it holds
$$
R_{N,\delta}^{(0)}(W^{\mu}_{s},\Psi_N, L_{2,\omega}, \ell_2)
\asymp \delta
\asymp N^{-\mu-1/s+1/2} .
$$
The order-optimal bounds are implemented by the method (\ref{ModVer1}).
\end{theorem}

Theorem \ref{2Sum_L} can be proven in the same way as Theorem \ref{Th5.5}.

\vskip 2mm

\begin{remark} \label{norma_L}
As follows from Theorems \ref{Sum_L}, \ref{2Sum_L}
for $p=2$ and any $s\ge 2$, the summation problem in $L_{2,\omega}$-metric is well-posed.
This means that for the indicated values of $s$, the problem is also well-posed for any $1\le p\le 2$ and
in any weaker metric, for example in $L_q$, $L_{q,\omega}$, $1\le q\le 2$,
where by $L_{q,\omega}$ we mean, as usual, the space of real-valued functions $f(t)$ with the norm
$\|f\|_{q,\omega} ;= \left(\int_{-1}^1 \omega(t) |f(t)|^q\, dt\right)^{1/q} < \infty$.
\end{remark}
		
\section{Computational experiments}  \label{num_exp}

		Now we provide some numerical illustrations for our theoretical results.
		We perform a modeling in Mathlab 2022a for functions with different smoothness.
		% To demonstrate the effectiveness of the proposed method for recovering higher order derivatives some numerical experiments were carried out. The calculations were performed on a computer with a 4-core Intel Core i5 processor and 16 GB memory in the  mathematical modelling environment MATLAB 2022a.
		
		\subsection{Example 1}
		
		In this example we consider the function  $f_1(t) = c(1-t^2)^{5/2}$ with $c=1/497.$
		%\begin{equation}\label{Exampl1}
		%\end{equation}
		It is easy to see that  $\|f_1\|_{s,\mu}\approx 1$ for $\mu = 5.4$ and $s=2$.
		
		For our experiment we take the perturbed data with random noise, which is simulated  by $\mathcal{F}^\delta=\mathcal{F}+c*\mbox{randn(size}(\mathcal{F})) \delta$, where $c=|\mathcal{F}|/(2\|\mathcal{F}\|_{\ell_2})$,
		randn and size are standard functions of the MATLAB system and $\mathcal{F}$ is a  vector of exact Fourier-Chebyshev coefficients.
		
		%The simulation of the noise in the input data was done in following way. A random noise adds to the values of the Fourier-Chebyshev coefficients. The noise is generated by the $\mbox{randn(size}({\ f})) \delta$
		%
		%W=randn(size(F)).*abs(F);
		%%     alpha=delta/(2*sqrt(sum(F.^2)));
		%%     V=alpha*W;
		
		%	\item  the values of the Fourier-Legendre coefficients recover  by  the quadrature trapezoid formula  on a uniform grid with a step $h$   so that condition (1.2) is satisfied for a given $\delta$.
		%	
		
		%	 Внесение в значения коэффициентов Фурье-Лежандра случайного шума, который генерируется командой $\mbox{randn(size}(F)) \delta$, где randn и size -- стандартные функции системы  MATLAB, а $F$ -- матрица точных значений коэффициентов Фурье-Лежандра;
		%%	\item   Восстановление значений коэффициентов Фурье-Лежандра по квадратурной формуле трапеций на равномерной сетке с шагом $hi$, так чтобы выполнялось условие  (1.2) для заданного  $\delta$.
		%	
		%\end{itemize}	

		Numerical experiments were carried out for the following error levels: $\delta= 10^{-4}, 10^{-5}, 10^{-6}$.
		Table \ref{tbl1}  presents  results of approximation of $f^{(1)}_1$ and  $f^{(2)}_1$ by the truncation method (\ref{ModVer}).
		The tables contain error estimations  in the  metrics of $L_{2,\omega}$ (column ErrorL2) and $C$ (column ErrorC)  and by  $N$
		we indicate the highest  degree of used Chebyshev polynomials.
		
		%Also in Table \ref{tbl2} $h$ means the step size in the quadrature formula.
		
		The figure \ref{Fig1} and \ref{Fig2} show results for $f^{(1)}_1$ and $f^{(2)}_1$ with different noise levels, respectively.

		\begin{table}[h!]
			\centering
			\caption{ Recovering derivatives $f^{(1)}_1$ and $f^{(2)}_1$   for random noise }
			\label{tbl1}
			\begin{tabular}{|c|c|c|c|c|c|c|}
				\hline
				
				& \multicolumn{3}{c|}{  Derivative $f^{(1)}_1$ }    	& \multicolumn{3}{c|}{  Derivative $f^{(2)}_1$ }
				\\ \hline
				$\delta$ & $10^{-4}$ & $10^{-5}$ & $10^{-6}$ & $10^{-4}$ & $10^{-5}$ & $10^{-6}$  \\ \hline
				ErrorL2        &  $5.8\cdot 10^{-5} $       &  $2 \cdot 10^{-5} $  &  $5.3  \cdot 10^{-6} $ &  $ 0.0068 $       &  $0.0014 $  &  $0.001 $ \\ \hline
				ErrorC        &  $2.3 \cdot 10^{-4} $       &  $9 \cdot 10^{-5} $  &  $2.6 \cdot10^{-5} $  &  $0.0221$       &  $0.0066$  &  $0.0051$  \\ \hline
				$N$  &  6     & 9   & 12 &  6     & 9   & 12   \\ \hline
				%		 $\card(\Gamma_{n})$  &  52      & 80    & 106    \\ \hline
			\end{tabular}
		\end{table}
		
		%\begin{table}[h!]
		%	\centering
		%	\caption{ The results of recovering derivative $f^{(2)}_1$ for random noise }
		%	\label{tbl2}
		%	\begin{tabular}{|c|c|c|c|}
			%		\hline
			%		%	& \multicolumn{3}{c|}{ Noise level }               \\ \hline
			%		$\delta$ & $10^{-4}$ & $10^{-5}$ & $10^{-6}$  \\ \hline
			%		ErrorL2        &  $ 0.0068 $       &  $0.0014 $  &  $0.001 $  \\ \hline
			%		ErrorC        &  $0.0221$       &  $0.0066$  &  $0.0051$  \\ \hline
			%		$n$  &  6    & 9   & 12    \\ \hline
			%		%		 $\card(\Gamma_{n})$  &  52      & 80    & 106    \\ \hline
			%	\end{tabular}
		%\end{table}

		\begin{figure}[h!]
			\begin{minipage}[h]{0.3\linewidth}
				\center{\includegraphics[width=1\linewidth]{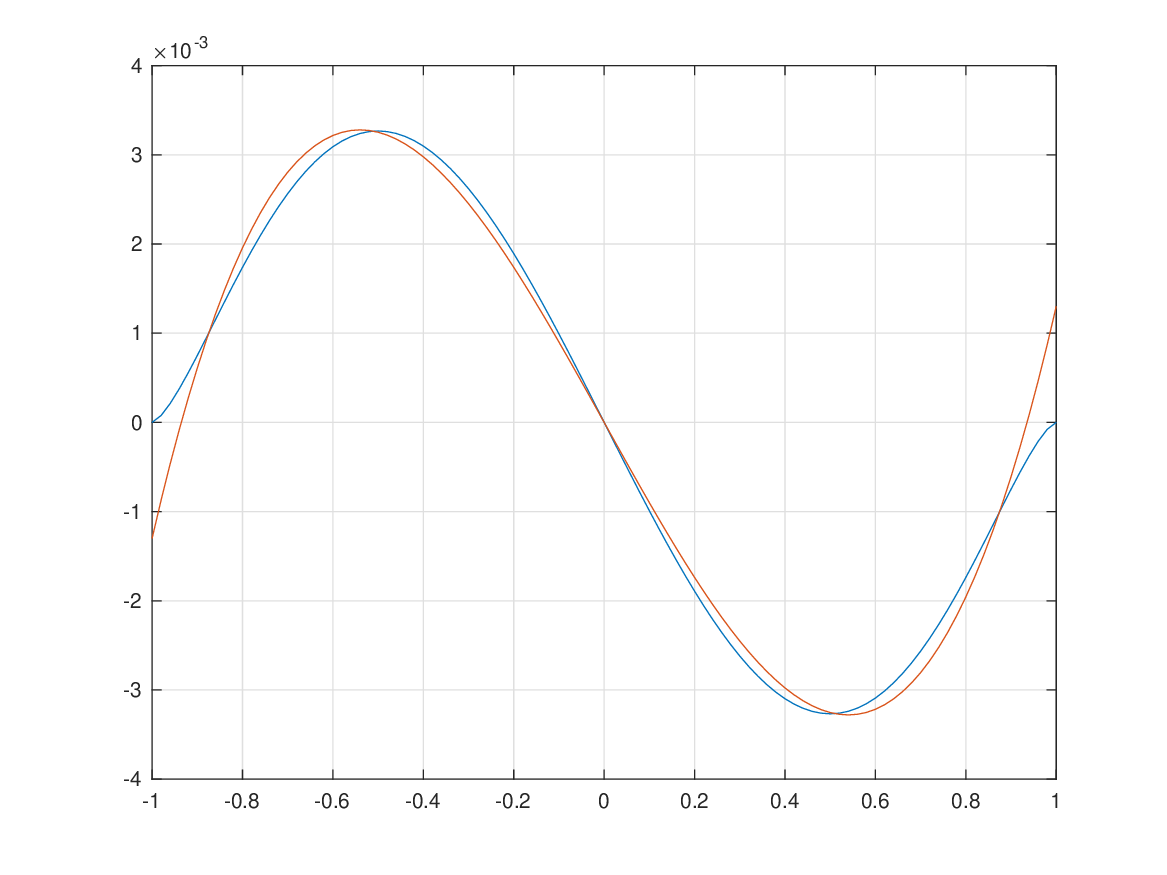} \\ a)}
			\end{minipage}
			\begin{minipage}[h]{0.3\linewidth}
				\center{\includegraphics[width=1\linewidth]{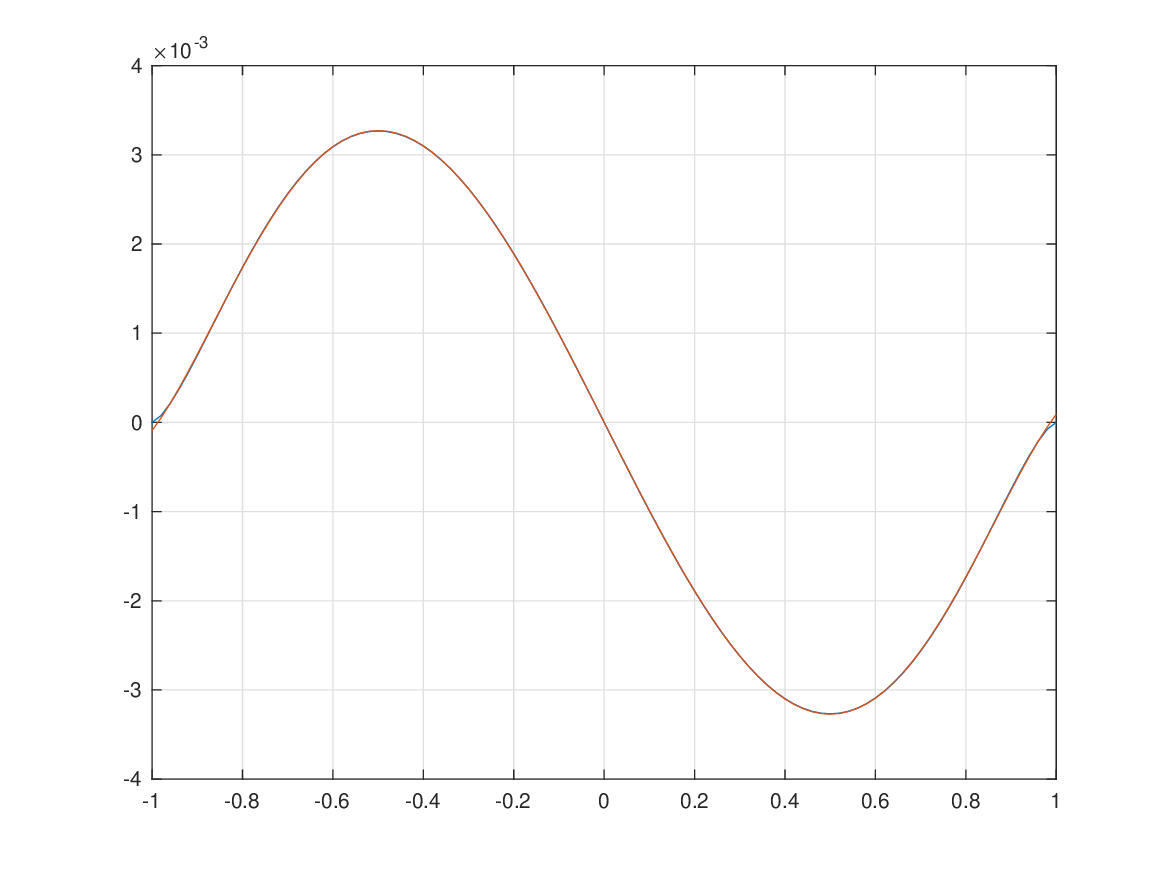} \\ b)}
			\end{minipage}
			\hfill
			\begin{minipage}[h]{0.3\linewidth}
				\center{\includegraphics[width=1\linewidth]{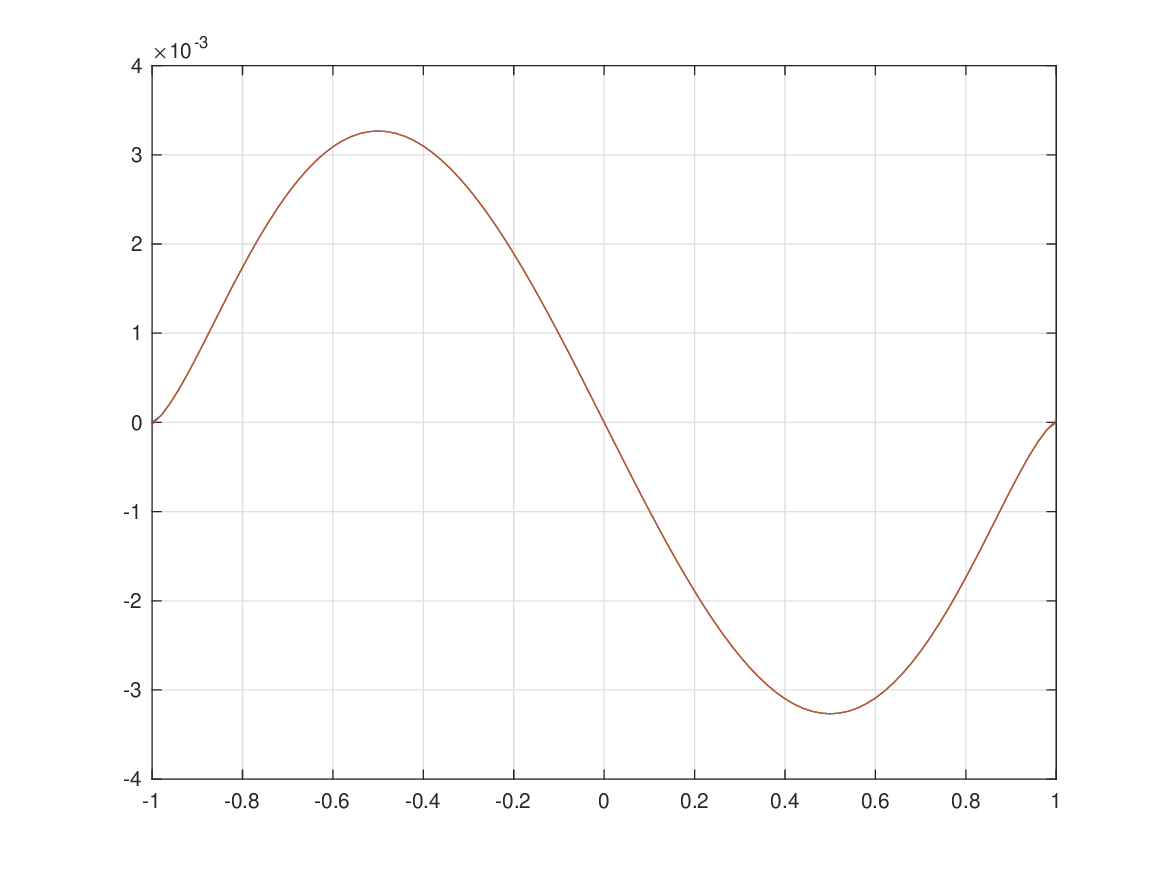} \\ c)}
			\end{minipage}
			%\begin{minipage}[h]{0.5\linewidth}
			%	\center{\includegraphics[width=1\linewidth]{Splain_RN_10_8.eps} \\ d)}
			%\end{minipage}
			\caption{Recovery of the derivative $f^{(1)}_1$  with  random noise in the  input data. Approximation to $f^{(1)}_1$ for    $\delta= 10^{-4}$ (Fig.  a) ),    for $\delta= 10^{-5}$ (Fig.b) ) and $\delta= 10^{-6}$ (Fig.  c) ) }
			\label{Fig1}
		\end{figure}

		\begin{figure}[h!]
			\begin{minipage}[h]{0.3\linewidth}
				\center{\includegraphics[width=1\linewidth]{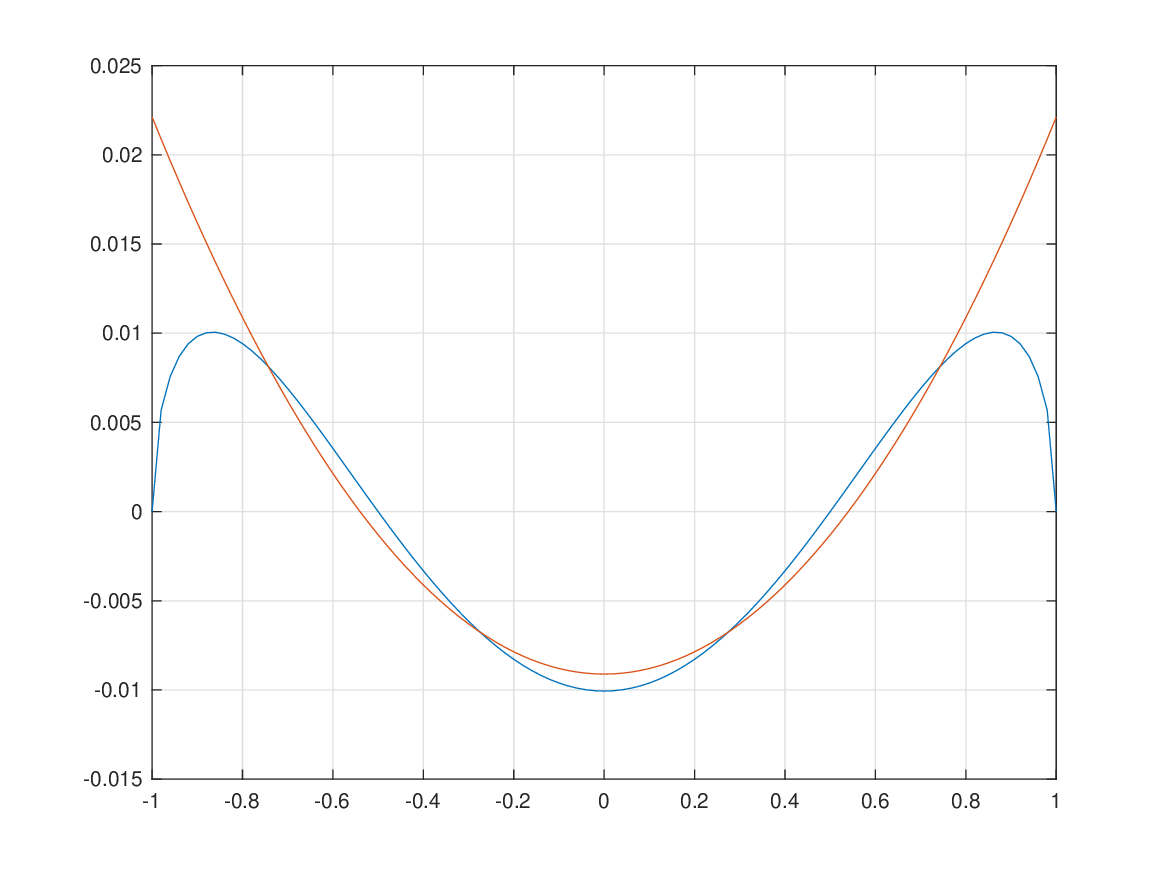} \\ a)}
			\end{minipage}
			\begin{minipage}[h]{0.3\linewidth}
				\center{\includegraphics[width=1\linewidth]{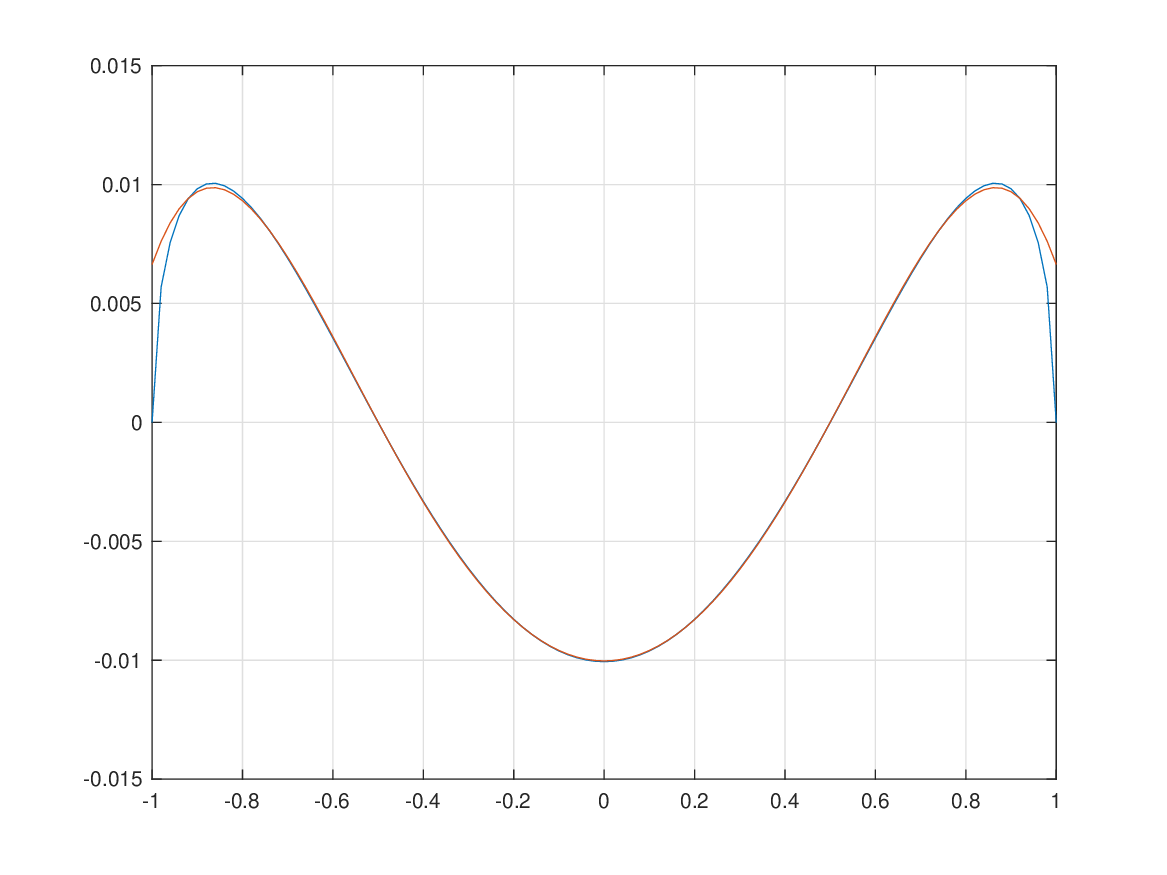} \\ b)}
			\end{minipage}
			\hfill
			\begin{minipage}[h]{0.3\linewidth}
				\center{\includegraphics[width=1\linewidth]{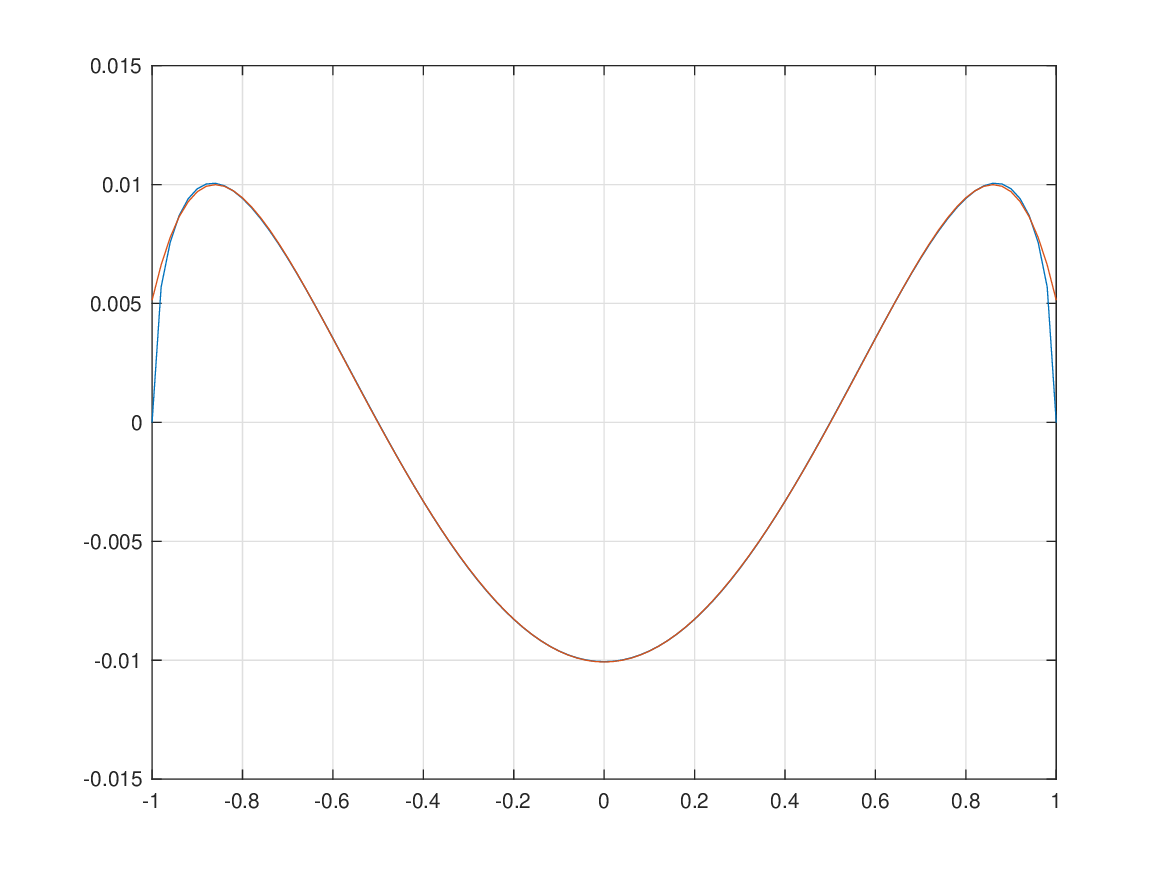} \\ c)}
			\end{minipage}
			%\begin{minipage}[h]{0.5\linewidth}
			%	\center{\includegraphics[width=1\linewidth]{Splain_RN_10_8.eps} \\ d)}
			%\end{minipage}
			\caption{Recovery of the derivative $f^{(2)}_1$  with  random noise in the  input data . Approximation to $f^{(2)}_1$ for    $\delta= 10^{-4}$ (Fig.  a) ),    for $\delta= 10^{-5}$ (Fig. b) ) and $\delta= 10^{-6}$ (Fig.  c) ) }
			\label{Fig2}
		\end{figure}

		%As can be seen from the graphs and table, the experimental results approved the theoretical results (see, Theorems \ref{Th2} and %\ref{Th1}).
		%for both  order of derivatives we have good approximations with minimal computational efforts.   At the same time, increasing the order of differentiation leads to increased instability of the problem and  loos of the accuracy of the approximation.

		%At the same time, applying  the quadrature formula expands the area of using the proposed method in computational problems, especially in the situation when the input data are given in the form of a set of function values at the grid nodes.
		
		\subsection{Example 2}
		
		Now we take the function $f_2(t) = c|t|$ with $c=1/7$.
		%\begin{equation}\label{Exampl1}
		%\end{equation}
		It is easy to see that  $\|f_2\|_{s,\mu}\approx 1$ for $\mu = 2$ and $s=2$.
		As follows from our theoretical results, for $f_2$ only the summation problem is correct.
		
		Just as in the previous example, for our experiment we take the Fourier-Chebyshev coefficients perturbed by random noise
		with noisy levels $\delta= 10^{-2}, 10^{-3}, 10^{-4}$.  In Table \ref{tbl3} and Figure \ref{Fig3} are given
		the corresponding numerical results.
		\begin{table}[h!]
			\centering
			\caption{ The results of summation $f_2$ for random noise }
			\label{tbl3}
			\begin{tabular}{|c|c|c|c|}
				\hline
				%	& \multicolumn{3}{c|}{ Noise level }               \\ \hline
				$\delta$ & $10^{-2}$ & $10^{-3}$ & $10^{-4}$  \\ \hline
				ErrorL2        &  $0.0021 $       &  $0.00045 $  &  $0.0002 $  \\ \hline
				ErrorC        &  $0.16 $       &  $0.0314 $  &  $0.024 $  \\ \hline
				$N$  &  8      & 29   & 49   \\ \hline
				%		$\card(\Gamma_{n})$  &  52      & 80    & 106    \\ \hline
			\end{tabular}
		\end{table}

		\begin{figure}[h!]
			\begin{minipage}[h]{0.3\linewidth}
				\center{\includegraphics[width=1\linewidth]{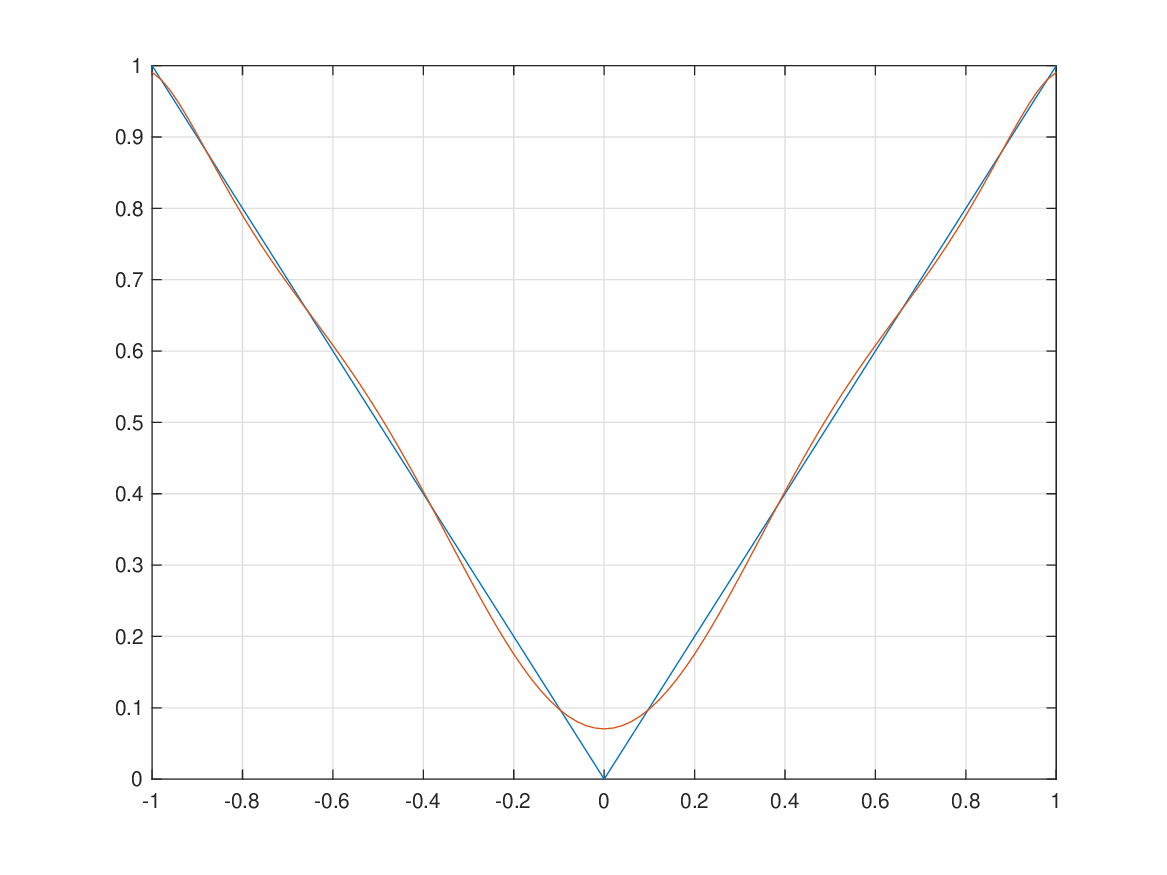} \\ a)}
			\end{minipage}
			\begin{minipage}[h]{0.3\linewidth}
				\center{\includegraphics[width=1\linewidth]{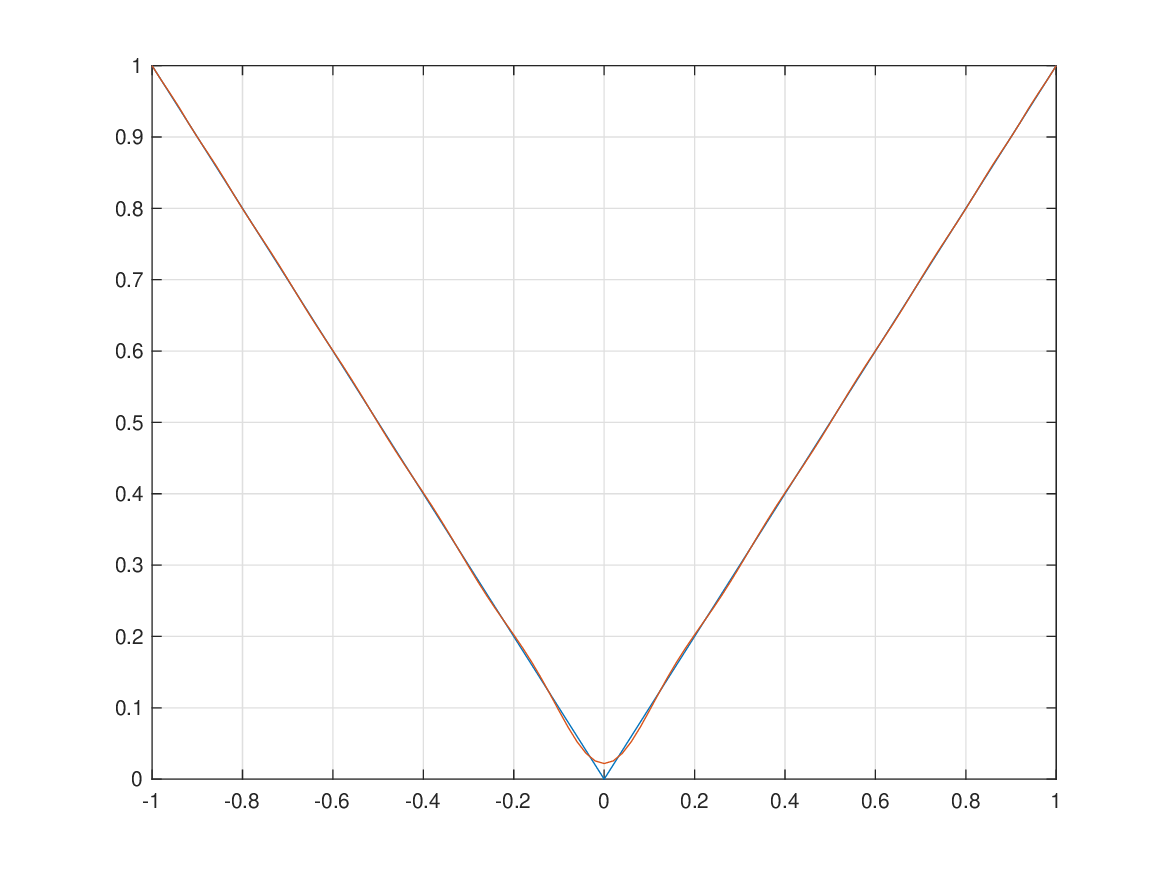} \\ b)}
			\end{minipage}
			%\hfill
			\begin{minipage}[h]{0.3\linewidth}
				\center{\includegraphics[width=1\linewidth]{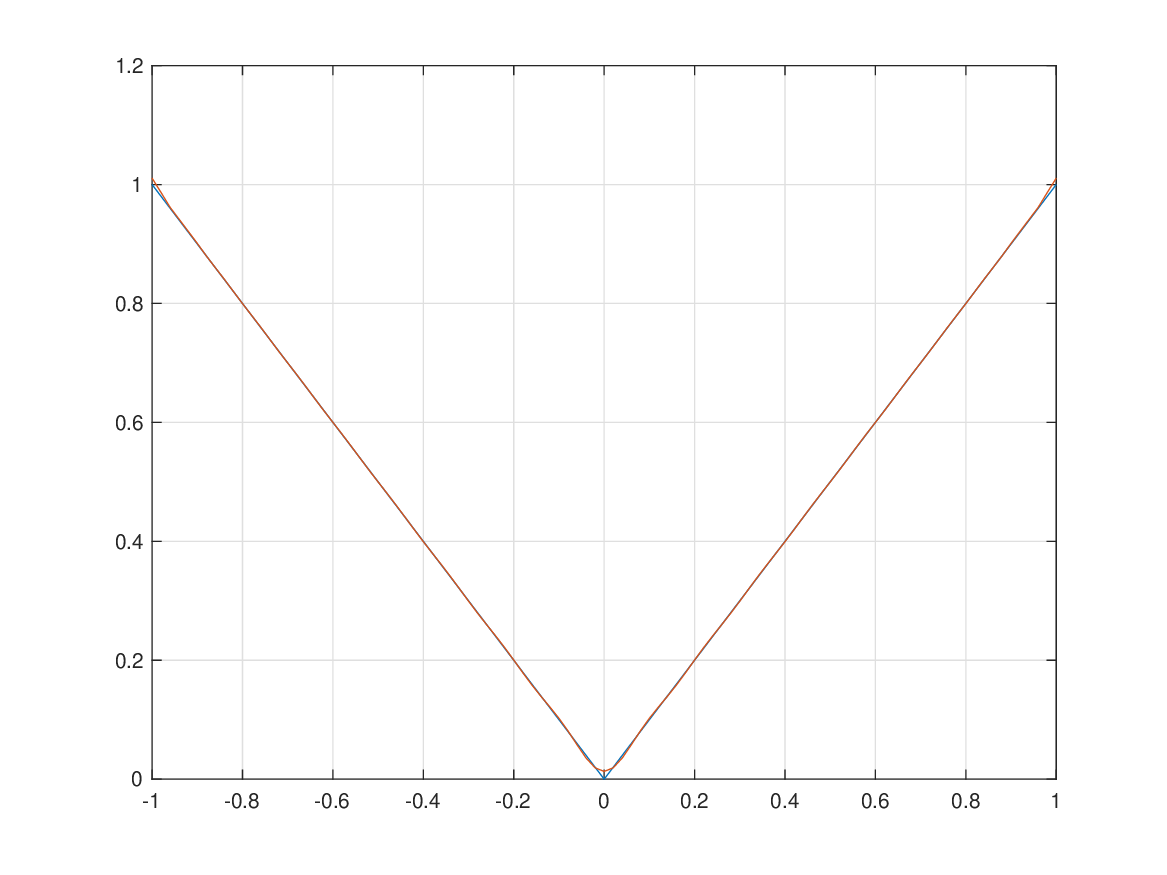} \\ c)}
			\end{minipage}
			%\begin{minipage}[h]{0.5\linewidth}
			%	\center{\includegraphics[width=1\linewidth]{Anal_Qf_10_8.eps} \\ d)}
			%\end{minipage}
			\caption{Summation of $f_2$  with  random noise in the  input data for    $\delta= 10^{-2}$ (Fig.  a) ),
				for $\delta= 10^{-3}$ (Fig. b) ) and $\delta= 10^{-4}$ (Fig.  c) ) }
			\label{Fig3}
		\end{figure}
		
		\subsection{Example 3}
		
		Let's test the method (\ref{ModVer}) when recovering the second derivative of an analytic function.
		We consider the function $ f_3(t)=c\, t\, \sin(t\, \pi/2 )$ with  $c=1/1580$.
		For our calculation we take $\mu= 6.5$, $s=2$ and evaluate the effectiveness of numerical recovering
		$f_3^{(2)}$ with 2 types of perturbations.
		Namely, we simulate noise in two ways, firstly, in a random way, as it was done in Example 1, and secondly,
		by applying a quadrature formula. As such a formula, we will take the Clenshaw-Curtis quadrature,
		according to which the Chebyshev-Fourier coefficients  (see \cite{MasHand}) for considered weight function are calculated as follows
		%	\begin{equation}\label{CC_quadrature}
			$$
			\langle  f_3, T_k \rangle
			%= \int_{-1}^{1} \frac{f_3(t) T_k(t)}{\sqrt{1-t^2}} dt
			\approx \frac{\pi}{n} \Big(\frac{f_3(t_0)T_k(t_0)}{2}+\frac{f_3(t_n)T_k(t_n)}{2}+\sum_{j=1}^{n-1}f_3(t_j)T_k(t_j)\Big),
			$$
			%	\end{equation}
where  $t_j = \cos\frac{j\pi}{n},\, j=0,\ldots,n$ are the zeros of $(1-t^2)U_{n-1}(t)$ and
$U_n$ are second-kind Chebyshev polynomials.
Here, the parameter of interpolation, $n$, is chosen
to provide a $\delta$-perturbation of the input data in the metric $\ell_2$.
This condition is verified using error estimation for the Clenshaw-Curtis quadrature (see, e.g., \cite{XiangBorne_2012}).
As is known, this choice of nodes $\{t_j\}$ is the most economical in terms of computational expenses
without reducing the accuracy of the approximation.
		%	Also formula (\ref{CC_quadrature}) is nearly equivalent to the   discrete Chebyshev transform for function $f$ in point $t_k$ except for the multiplier.
		%The error estimation of the Clenshaw-Curtis quadrature for can be found in \cite{XiangBorne_2012}.
		%We note that  values of discrete Chebyshev transforms  proportional to the Fouier-Chebyshev coefficients  (\ref{CC_quadrature}).
		%The nodes in (\ref{CC_quadrature}) are taken in such a way that the orders of noise are $\delta= 10^{-5}, 10^{-6}, 10^{-7}$.
		The quantity of nodes is involved in the approximation indicated in the column "n".
		
		Tables \ref{tbl4} and \ref{tbl5} show the results of approximation  of $f_3^{(2)}$ by the truncation method (\ref{ModVer})
		with both types of perturbations.
		The Figure \ref{Fig5} show the approximations  to $f_3^{(2)}$ built by input data with the noise from Clenshaw-Curtis quadrature.
		
		As can be seen from the Tables \ref{tbl4} and \ref{tbl5}, we achieve acceptable accuracy for both type of noise.
		But applying  the quadrature formula expands the area of using the proposed method in calculations,
		especially in the situation when the input data is given in the form of a set of function values at grid nodes.
		
		%The Fourier-Legendre coefficients were calculated using the quadrature trapezoid formula for $h = 4\cdot 10^{-4}, 10^{-4} , 4\cdot10^{-5} $, which in turn according to formula (1.2) matches $\delta= 10^{-6}, 10^{-7}, 10^{-8}$ in the order.
		
		%Продемонстрируем наш метод для приближения $F^(2,2)$ аналитической функции. Рассмотрим функцию $ F(x,y)=1/C(2-(2x-1)^2)^2 \cos(4y)$ которая уже была использована в \cite{WW2005}. For $C=43940129$ it is easy to see that  $\|F\|_{2,\mu} \approx 1$ for $\mu=6.$ При этом норма $\|F^{(2,2)}\|_{2}\approx 4,1\cdot 10^{-3}.  $
		%Коэффициенты Фурье-Лежандра были подсчитаны по квадратурной формуле трапеций для $h = 4\cdot 10^{-4}, 10^{-4} , 4\cdot10^{-5} $, что в свою очередь  согласно формуле (1.2) соответствует по порядку  $\delta= 10^{-6}, 10^{-7}, 10^{-8}$.

		\begin{table}[h!]
			\centering
			\caption{ Recovering derivative $f_3^{(2)}$ with random noise  }
			\vspace{2mm}
			
			\label{tbl4}
			\begin{tabular}{|c|c|c|c|}
				\hline
				%	& \multicolumn{3}{c|}{ Noise level }               \\ \hline
				$\delta$ & $10^{-5}$ & $10^{-6}$ & $10^{-7}$  \\ \hline
				Error L2        &  $2 \cdot 10^{-4}  $       &  $ 8.69\cdot 10^{-6}   $  &  $1.59  \cdot 10^{-7}   $  \\ \hline
				Error C        &  $5 \cdot 10^{-4} $       &  $2.79 \cdot 10^{-5}   $  &  $6  \cdot 10^{-7}   $  \\ \hline
				N  &  6     & 7   & 10    \\ \hline
				
			\end{tabular}
		\end{table}
		
%		\begin{figure}[h!]
%			\begin{minipage}[h]{0.5\linewidth}
%				\center{\includegraphics[width=1\linewidth]{Der2_anal_Rand10-5.eps} \\ a)}
%			\end{minipage}
%			\begin{minipage}[h]{0.5\linewidth}
%				\center{\includegraphics[width=1\linewidth]{Der2_anal_Rand10-6.eps} \\ b)}
%			\end{minipage}
%			\hfill
%			\begin{minipage}[h]{0.5\linewidth}
%				\center{\includegraphics[width=1\linewidth]{Der2_anal_Rand10-7.eps} \\ c)}
%			\end{minipage}
%			%\begin{minipage}[h]{0.5\linewidth}
%			%	\center{\includegraphics[width=1\linewidth]{Splain_RN_10_8.eps} \\ d)}
%			%\end{minipage}
%			\caption{Recovery of the derivative $f_3^{(2)}$  with  random noise in the  input data. Approximation to $f_3^{(2)}$ for
%				$\delta= 10^{-5}$ (Fig.  a) ),  $\delta= 10^{-6}$ (Fig. b) ) and $\delta= 10^{-7}$ (Fig.  c)) }
%			\vspace{2mm}
%			\label{Fig4}
%		\end{figure}

		\begin{table}[h!]
			\centering
			\caption{ Recovering derivative $f^{(2)}_3$ with noise from Clenshaw-Curtis quadrature  }
			\label{tbl5}
			\begin{tabular}{|c|c|c|c|}
				\hline
				%	& \multicolumn{3}{c|}{ Noise level }               \\ \hline
				$\delta$ & $10^{-5}$ & $10^{-6}$ & $10^{-7}$  \\ \hline
				Error L2        &  $2.25 \cdot 10^{-4}  $       &  $ 5.9\cdot 10^{-6}   $  &  $2.32  \cdot 10^{-7}   $  \\ \hline
				Error C        &  $5.6 \cdot 10^{-4} $       &  $1.9 \cdot 10^{-5}   $  &  $8.3  \cdot 10^{-7}   $  \\ \hline
				N  &  7     & 9 & 13    \\ \hline
				n  &  6     & 9  & 13    \\ \hline
				
			\end{tabular}
		\end{table}

		\begin{figure}[h!]
			\begin{minipage}[h]{0.3\linewidth}
				\center{\includegraphics[width=1\linewidth]{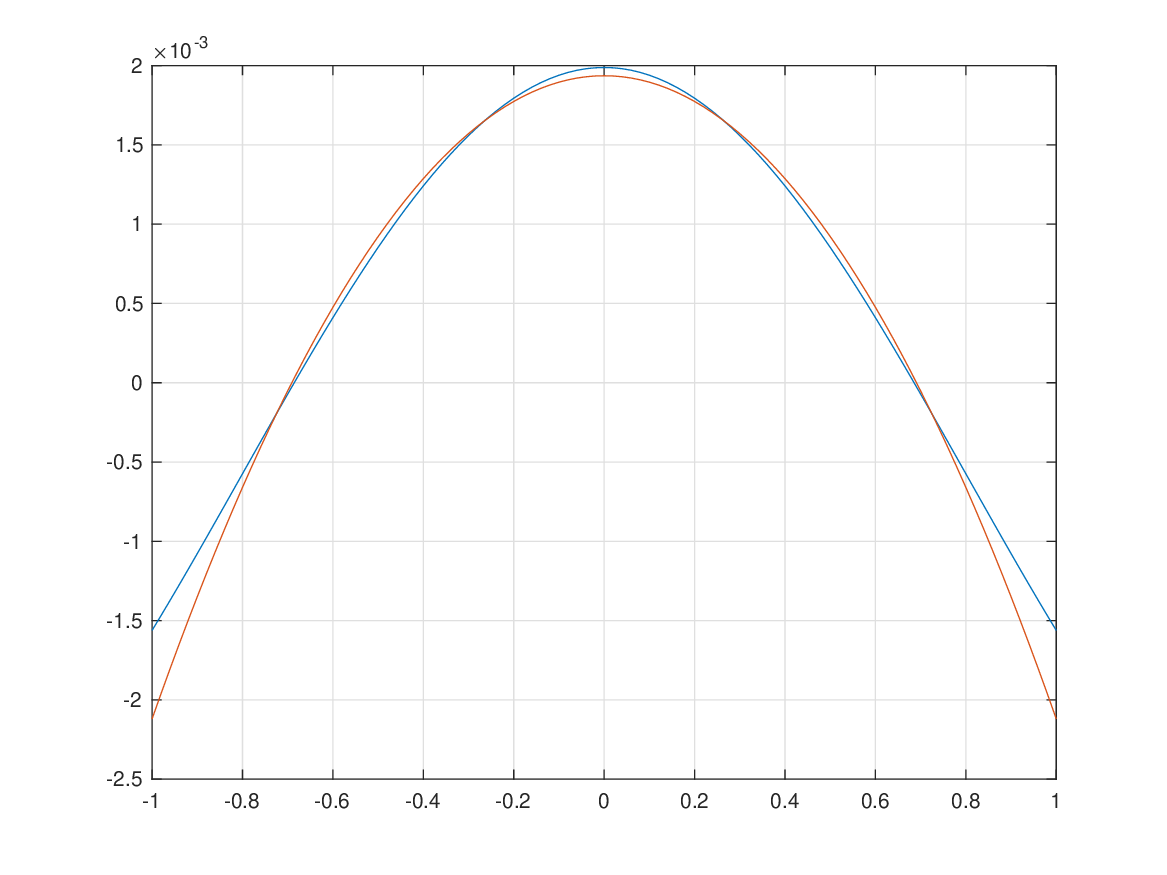} \\ a)}
			\end{minipage}
			\begin{minipage}[h]{0.3\linewidth}
				\center{\includegraphics[width=1\linewidth]{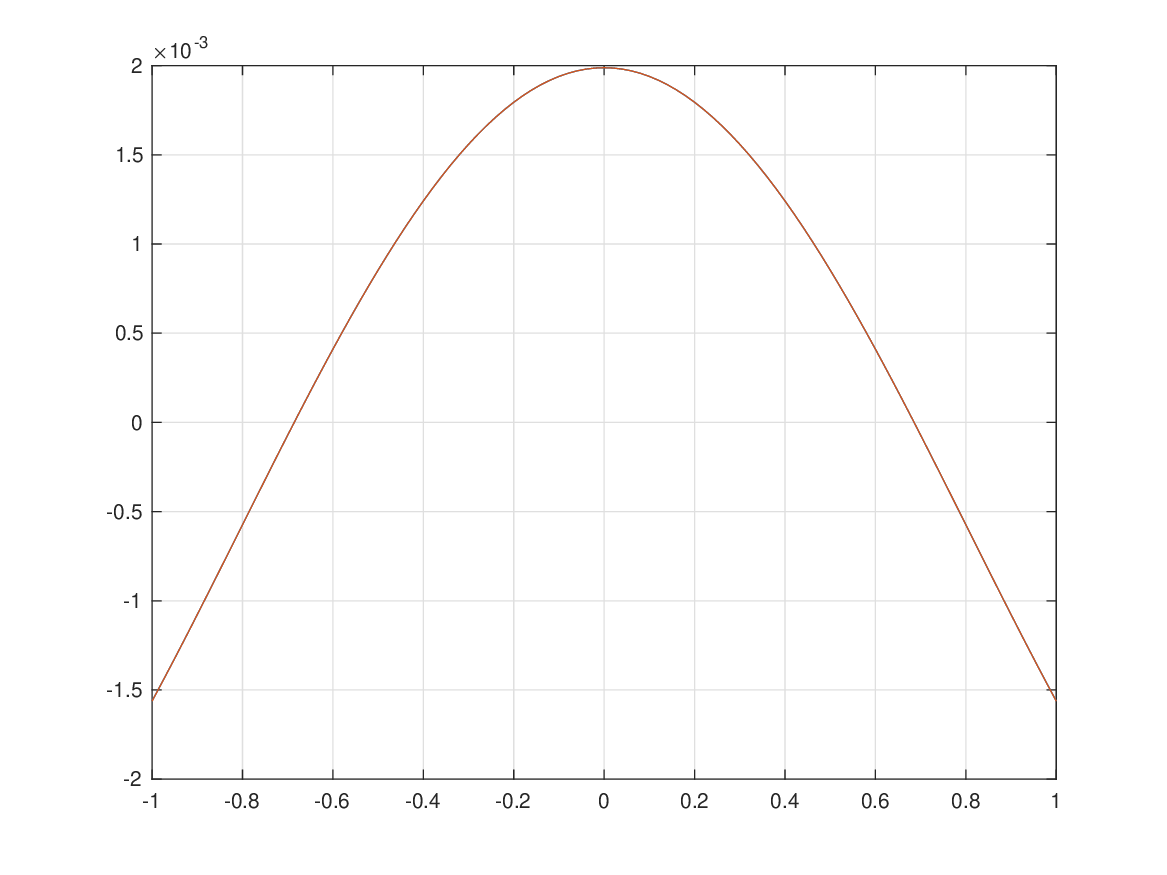} \\ b)}
			\end{minipage}
			\hfill
			\begin{minipage}[h]{0.3\linewidth}
				\center{\includegraphics[width=1\linewidth]{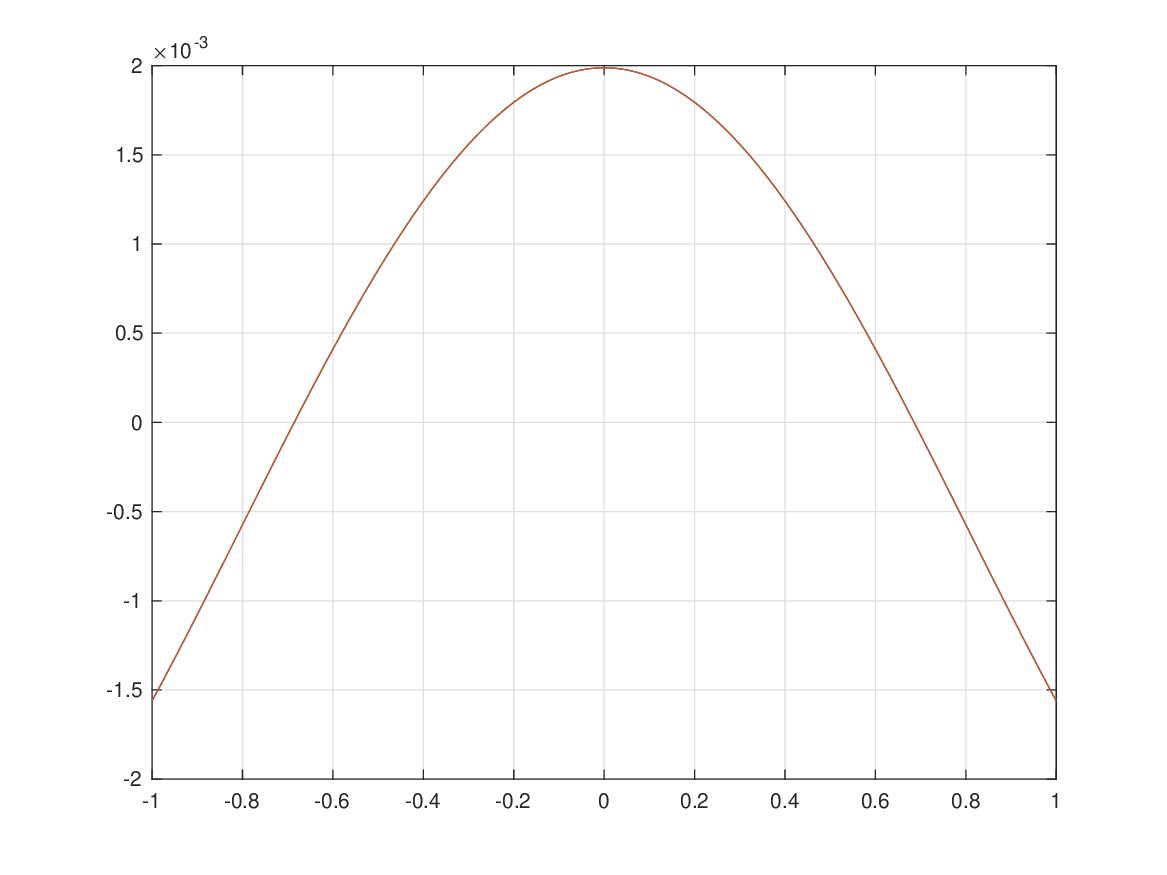} \\ c)}
			\end{minipage}
			%\begin{minipage}[h]{0.5\linewidth}
			%	\center{\includegraphics[width=1\linewidth]{Splain_RN_10_8.eps} \\ d)}
			%\end{minipage}
			\caption{Recovery of the derivative $f_3^{(2)}$  with noise from Clenshaw-Curtis quadrature. Approximation to $f_3^{(2)}$ for    $\delta= 10^{-5}$ (Fig.  a)),  $\delta= 10^{-6}$ (Fig. b)) and $\delta= 10^{-7}$ (Fig.  c)) }
			\vspace{2mm}
			\label{Fig5}
		\end{figure}

%		\section*{ Acknowledgement }

\section*{ Acknowledgments and Disclosure of Funding}
We are very grateful to Prof. Oleg Davydov for discussion and suggestion
	which improved the article a lot.

This project has received funding through the MSCA4Ukraine project,
which is funded by the European Union (ID number 1232599).
In addition, the first named author is supported by scholarship of University of L\"ubeck.
Also, the authors acknowledge partial financial support due to the project "Mathematical modelling of complex dynamical systems and processes caused by the state security" (Reg. No. 0123U100853).

\bibliography{ND_Cheb_RDer}

\end{document}